\documentclass[a4paper,12pt]{article}

\usepackage{amsmath}
\usepackage{mathtools}
\usepackage{amsfonts}
\usepackage{amssymb}
\usepackage{amsthm}
\usepackage{graphicx}
\usepackage{psfrag}
\usepackage{subfigure}

\usepackage[english]{babel}
\usepackage[latin9]{inputenc}
\usepackage[T1]{fontenc}
\usepackage{amssymb}
\usepackage{xcolor}
\usepackage[varumlaut]{yfonts}
\usepackage{amsmath}
\usepackage{mathtools}
\usepackage{amsopn}
\usepackage{amsthm}
\usepackage{ulem}
\usepackage{setspace}
\usepackage{esint}
\usepackage{graphicx}
\usepackage{babel}
\usepackage{relsize}
\usepackage{float}
\usepackage[labelfont={bf}]{caption}
\usepackage{nicefrac}
\usepackage{paralist}
\usepackage{verbatim}
\usepackage{enumitem}
\usepackage{soul}

\newcommand{\dd}{\ensuremath{\, \mathrm{d}}}

\newtheoremstyle{definition}%
{}{}%
{\itshape}{}%
{\bfseries}{}
{\newline}{}

\theoremstyle{definition}
\newtheorem{definition}{Definition}[section]
\newtheorem{satz}[definition]{Theorem}

\theoremstyle{plain}
\newtheorem{rem}[definition]{Remark}

\newtheorem{lem}[definition]{Lemma}

\newtheorem{claim}[definition]{Claim}

\newtheorem{prop}[definition]{Proposition}

\numberwithin{equation}{section}  

\begin{document}
\begin{titlepage}
\title{A Blow-up Criterion for the Curve Diffusion Flow with a Contact Angle}
\author{  Helmut Abels\footnote{Fakult\"at f\"ur Mathematik,  
Universit\"at Regensburg,
93040 Regensburg,
Germany, e-mail: {\sf helmut.abels@mathematik.uni-regensburg.de}}~ and Julia Butz\footnote{Fakult\"at f\"ur Mathematik,  
Universit\"at Regensburg,
93040 Regensburg,
Germany, e-mail: {\sf julia4.butz@mathematik.uni-regensburg.de}}}
\end{titlepage}
\maketitle
\begin{abstract}
We prove a blow-up criterion in terms of an $L_2$-bound of the curvature for solutions to the curve diffusion flow if the maximal time of existence is finite. In our setting, we consider an evolving family of curves driven by curve diffusion flow, which has free boundary points supported on a line. The evolving curve has fixed contact angle $\alpha \in (0, \pi)$ with that line and satisfies a no-flux condition. The proof is led by contradiction: A compactness argument combined with the short time existence result enables us to extend the flow, which contradicts the maximality of the solution. 
\end{abstract}

\noindent{\bf AMS-Classification:} 53C44, 35K35, 35K55
\section{Introduction and Main Result} \label{intro}

The curve diffusion flow is the one-dimensional version of the surface diffusion flow, which describes the motion of interfaces in the case that it is governed purely by diffusion within the interface. It was originally derived by Mullins to model the development of surface grooves at the grain boundaries of a heated polycrystal in 1957, see \cite{mullins}. Moreover, the flow is related to the Cahn-Hilliard equation for a degenerate mobility, see e.g.\ 
\cite{cahnelliottnovick} 
or \cite{garckenovick}. \\

Our goal is to establish a blow-up criterion for the curve diffusion flow with a contact angle in terms of the curvature. More precisely, we consider the time dependent family of regular open curves $\{\Gamma_t\}_{t \geq 0}$ moving according to 
\begin{align}
V = - \partial_{ss} \kappa_{\Gamma_t}  \hspace{3cm} \textrm{ on } \Gamma_t, t > 0, \label{0}
\end{align}
where $V$ is the scalar normal velocity, $\kappa_{\Gamma_t}$ is the scalar curvature of $\Gamma_t$, and $s$ denotes the arc length parameter. We complement the evolution law with the boundary conditions 
\begin{align}
\partial \Gamma_t &\subset \mathbb{R} \times \{0\} && \textrm{ for } t > 0, \label{2blow} \\
\measuredangle \left({n}_{\Gamma_t},  \begin{pmatrix} 0 \\ -1\end{pmatrix} \right) &= \pi - \alpha && \textrm{ at } \partial \Gamma_t \textrm{ for } t > 0,  \label{4blow} \\
\partial_s \kappa_{\Gamma_t} &= 0  && \textrm{ at } \partial \Gamma_t \textrm{ for } t > 0, 
 \label{3blow}
\end{align}
where ${n}_{\Gamma_t}$ is the unit normal vector of $\Gamma_t$ and $\alpha \in (0, \pi)$. We give a sketch of the geometrical situation in Figure \ref{1}.
\begin{center}\vspace{0 cm}
	\scalebox{1}{
\begingroup%
  \makeatletter%
  \providecommand\color[2][]{%
    \errmessage{(Inkscape) Color is used for the text in Inkscape, but the package 'color.sty' is not loaded}%
    \renewcommand\color[2][]{}%
  }%
  \providecommand\transparent[1]{%
    \errmessage{(Inkscape) Transparency is used (non-zero) for the text in Inkscape, but the package 'transparent.sty' is not loaded}%
    \renewcommand\transparent[1]{}%
  }%
  \providecommand\rotatebox[2]{#2}%
  \ifx\svgwidth\undefined%
    \setlength{\unitlength}{250.5851675bp}%
    \ifx\svgscale\undefined%
      \relax%
    \else%
      \setlength{\unitlength}{\unitlength * \real{\svgscale}}%
    \fi%
  \else%
    \setlength{\unitlength}{\svgwidth}%
  \fi%
  \global\let\svgwidth\undefined%
  \global\let\svgscale\undefined%
  \makeatother%
  \begin{picture}(1,0.43231757)%
    \put(0,0){\includegraphics[width=\unitlength]{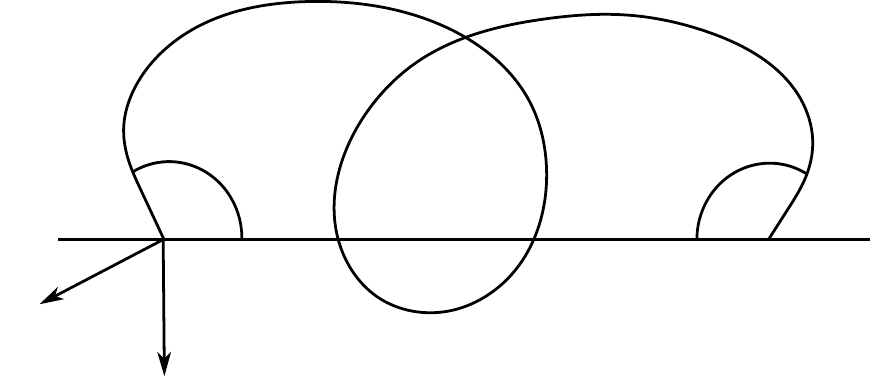}}%
    \put(0.85991725,0.18789509){\color[rgb]{0,0,0}\makebox(0,0)[b]{\smash{$\alpha$}}}%
    \put(0.48844155,0.29124704){\color[rgb]{0,0,0}\makebox(0,0)[lb]{\smash{$\Gamma_t$}}}%
    \put(0.20381141,0.1870888){\color[rgb]{0,0,0}\makebox(0,0)[lb]{\smash{$\alpha$}}}%
    \put(0.21255937,0.0648694){\color[rgb]{0,0,0}\makebox(0,0)[lb]{\smash{$ \begin{pmatrix} 0 \\ -1\end{pmatrix}$}}}%
    \put(-0.00043648,0.11280602){\color[rgb]{0,0,0}\makebox(0,0)[lb]{\smash{$\vec{n}_{\Gamma_t}$}}}%
  \end{picture}%
\endgroup%
}
	\captionof{figure}{Evolution by curve diffusion flow with $\alpha$-angle condition for $\alpha > \tfrac{\pi}{2}$.}
	\label{1}
\end{center}\vspace{0 cm}
Escher, Mayer, and Simonett gave numerical evidence that closed curves in the plane, which are moving according to \eqref{0} can develop singularities in finite time, cf.\ \cite{eschmaysim}. Indeed, for smooth closed curves driven by \eqref{0}, Chou provided a sharp criterion for a finite lifespan of the flow in \cite{chou}. Additionally, Chou, see \cite{chou}, and Dzuik, Kuwert, and Schätzle, see \cite{dziuk}, showed that if a solution has a maximal lifespan $T_{max} < \infty$, then the $L_2$-norm of the curvature with respect to the arc length parameter tends to infinity as $T_{max}$ is approached. Moreover, they gave a rate for the blow-up.\\ 

Before we state the main result, we give our notion of solution. We refer to Section \ref{pre} below for the definitions of the function spaces in the following.
\begin{definition} [Strong Solution, Maximal Solution] \label{strong}
	Let $f_0: \bar{I} \to \mathbb{R}^2$, $I := (0, 1)$, be a regular curve in $W_2^{4\left(\mu - \nicefrac{1}{2}\right)}(I; \mathbb{R}^2)$, $\mu \in \left(\frac{7}{8}, 1 \right]$. Furthermore, let it fulfill
	\begin{align}
	f_{0}(\sigma) &\in \mathbb{R} \times \{0\} && \text{ for } \sigma \in \{0, 1\}, \nonumber \\
	\measuredangle \left(n_{\Gamma_0}(\sigma),  \begin{pmatrix} 0 \\ -1\end{pmatrix} \right) &= \pi - \alpha && \text{ for } \sigma \in \{0, 1\},
	\label{initialdatum}
	\end{align}
	where $\Gamma_0 := f_0(\bar{I})$ and $\alpha \in (0, \pi)$ and $\measuredangle(u, v)$ denotes the angle between two vectors $u, v \in \mathbb{R}^2$. We call $f:[0, T) \times \bar{I} \rightarrow \mathbb{R}^2$, a \textbf{strong solution} of the curve diffusion flow, if the following holds true:
	\begin{enumerate}
		\item $f \in W^1_{2, \mu, loc} \left([0, T); L_{2} (I; \mathbb{R}^2) \right) \cap L_{2, \mu, loc} \left([0, T); W^{4}_{2} (I; \mathbb{R}^2) \right)$,
		where 
		\begin{align*}
		W^k_{2, \mu, loc} ([0, T); E) := \big\{&u: [0, T) \rightarrow E \textrm{ is strongly measurable }: u_{|(0,T')} \in W^k_{2, \mu} ((0, T'); E) \\ &\textrm{ for all } 0 < T' < T) \big\},
		\end{align*}
		\item $f$ fulfills the equations \eqref{0}-\eqref{3blow} and there exists a regular $C^1$-reparametrization $\varphi: [0, 1] \rightarrow [0, 1]$ such that $f_0(\varphi(\sigma)) = f(0, \sigma)$ for all $\sigma \in [0, 1]$,
		\item $f(t, \cdot)$ is for each $t \in [0, T)$ a regular parametrization of the curve $f(t, \bar{I})$.
	\end{enumerate}
	If $T$ is the largest time for a given $f_0$ such that there is a strong solution on $[0, T)$, we set $T_{max} = T$ and call it a \textbf{maximal solution} of curve diffusion flow. 
      \end{definition}
      Our first main result is:
\begin{satz} [Blow-up Criterion] \label{c}
	Let $f: [0, T_{max}) \times \bar{I} \rightarrow \mathbb{R}^2$, $I:=(0,1)$, $T_{max} < \infty$, be 
	a maximal solution of \eqref{0}-\eqref{3blow}.
	Then $\lim_{t \to T_{max}} \|\kappa [f(t)]\|_{L_2(0, \mathcal{L}[f(t)])} = \infty$. 
\end{satz}

The strategy of our proof is inspired by the blow-up criterion in \cite{dziuk}: By assumption, $f: [0, T_{max}) \times \mathbb{S}^1 \rightarrow \mathbb{R}^n$, $T_{max} < \infty$, is a smooth solution of \eqref{0} for closed curves, which cannot be extended in time. The authors assume, contrary to their claim, that $\|\kappa(t)\|_{L_2}$ is uniformly bounded in $t < T_{max}$. Here $\|\cdot\|_{L^2}$ denotes the $L_2$-norm with respect to the arc length parameter. Carrying on, they consider the normal component of the derivative, i.e.\ $\nabla_s \vec{\kappa} := \partial_s \vec{\kappa} - \langle \partial_s \vec{\kappa}, \tau \rangle,$
where $\langle \cdot, \cdot \rangle$ denotes the Euclidean inner product on $\mathbb{R}^n$ and $\tau$ is the unit tangent vector. Using the motion law, they obtain differential inequalities for $\|\nabla^m_s \vec{\kappa}(t)\|_{L_2}$ for all $t < T_{max}$ and for all $m \in \mathbb{N}$. By the curvature bound, they  iteratively establish bounds on $\|\nabla^m_s \vec{\kappa}(t)\|_{L_2}$ by Gagliardo-Nirenberg-type inequalities. Comparing the full arc length derivatives $\partial^m_s \vec{\kappa}(t)$ to the projected ones $\nabla^m_s \vec{\kappa}(t)$, they can prove bounds on the $L_2$-norms of the full spatial derivatives of the curvature vector for all $t < T_{max}$. This permits for an extension of the flow beyond $T_{max}$, which contradicts the maximality of the solution. 

The authors use the same approach also for showing global existence of solutions: Considering an $L_2$-gradient flow of an energy which provides a bound on $\|\kappa(t)\|_{L_2}$ the strategy allows for proving global existence of solutions and subconvergence results as $t \rightarrow \infty$, i.e.\ convergence of a subsequence. In \cite{dziuk}, Dziuk, Kuwert, and Schätzle also inspect the bending energy with length penalization for closed curves, which is for a smooth, regular $f : \mathbb{S}^1 \rightarrow \mathbb{R}^n$, $n\geq2$ given by
\begin{align}
\mathcal{B}[f] := \int_{\mathbb{S}^1} \left( \frac{1}{2} |\vec{\kappa}|^2 + \lambda \right) ds && \textrm{ for } \lambda \in \mathbb{R}. \label{B} 
\end{align}
They obtain by the previously described technique that for smooth, regular initial data the $L_2$-gradient flow of \eqref{B} with $\lambda \in [0,\infty)$ has a smooth global solution. In the case $\lambda > 0$, they deduce that it subconverges to an equilibrium after reparametrization to arc length and a suitable translation. Moreover, they give an analogous result for the $L_2$-gradient flow of \eqref{B} for $\lambda = 0$ with a length constraint.

The same strategy was also adapted to the case of open curves: In \cite{lin}, \cite{acquapozzi}, and \cite{acqualinpozzi}, the authors consider $L_2$-gradient flows of the bending energy, either with length penalization or with length constraint, for open curves. For a smooth, regular function $f : I \rightarrow \mathbb{R}^n$, $n\geq2$, $I$ a closed bounded interval, they look for different parameters $\xi$ and $\lambda$ at the energy
\begin{align}
\mathcal{E}[f] := \int_{I} \left( \frac{1}{2} |\vec{\kappa} - \xi|^2 + \lambda \right) ds && \textrm{ for } \xi \in \mathbb{R}^n \textrm{ and } \lambda \in \mathbb{R}. \label{E} 
\end{align}
Here, the vector $\xi$ is called spontaneous curvature, see \cite{acquapozzi}.

Lin proved a global existence result for the $L^2$-flow of \eqref{E} for fixed $\lambda \in \mathbb{R}^+$ and $\xi=0$ for open curves with clamped boundary conditions in \cite{lin}. The gradient flow is considered among curves with fixed boundary points and fixed tangent vectors at the boundary points. Additionally, the initial datum is supposed to be smooth with positive, finite length and satisfying certain compatibility conditions. Again, it is assumed that $f$
is a smooth solution which cannot be extended in time. The author controls $\nabla^m_t f (t)$ for all $t < T_{max}$, instead of $\nabla_s^m \vec{\kappa}(t)$ as in \cite{dziuk}, where $\nabla_t f := \partial_t f - \langle \partial_t f, \tau \rangle.$  
He obtains bounds on $\|\nabla^m_t f (t) \|_{L_2}$ for all $m \in \mathbb{N}$ in terms of $\nabla^p_s \vec{\kappa}(t)$, $p \in \mathbb{N}$, for all $t < T_{max}$. In contrast to the setting in \cite{dziuk}, attention has to be paid to the boundary terms, which occur due to integration by parts. Thus, the quantities $\nabla^m_t f (t)$ are a clever choice, as they vanish at the boundary points due to the boundary conditions. Additionally, from the global existence of the flow it is deduced that the family of curves subconverges after reparametrization by arc length to an equilibrium. Dall'Acqua, Pozzi, and Spener strengthened the result of Lin in \cite{lin} by showing that up to a time dependent reparametrization $\phi(t, \cdot): I \rightarrow I$, $t \in [0, \infty)$, the whole solution $f(t, \phi(t, \cdot))$ converges to a critical point of the energy in $L_2$ for $t \rightarrow \infty$, see \cite{acquapozzispener}.

In \cite{acquapozzi}, Dall'Acqua and Pozzi proved a global existence and subconvergence result for the $L^2$-flow of the energy \eqref{E} for $\lambda \in [0,\infty)$ and $\xi \in \mathbb{R}^n$, with fixed boundary points and such that the curvature vector equals the normal component of $\xi$ at the boundary. In this setting, the authors also control the quantities $\nabla^m_t f (t)$, but those do not vanish at the boundary and thus have to be analyzed carefully.

Moreover, Dall'Acqua, Lin, and Pozzi obtained an analogous result for the $L^2$-gradient-flow of \eqref{E} with $\xi = 0$, which is complemented with hinged boundary conditions, i.e.\ fixed boundary points and zero curvature at the boundary points, see \cite{acqualinpozzi}. Additionally, a special time dependent $\lambda$ is chosen to preserve the length of the curve during the flow. \\

The strategy of the proof of Theorem \ref{c} is similar: For a given maximal solution $f$, we assume, contrary to the claim, that $\|\kappa(t_l)\|_{L_2}$ is uniformly bounded for a sequence $(t_l)_{l \in \mathbb{N}}$ with $t_l \rightarrow T_{max}$ for $l \rightarrow \infty$. In contrast to the reasoning in \cite{dziuk} we do not work with smooth solutions: For $\alpha \in (0, \tfrac{\pi}{2}) \cup (\tfrac{\pi}{2}, \pi)$ the considered flow also has a tangential component, as the boundary points are attached to the $x$-axis during the flow. This makes the control of the boundary terms complicated. By just using the bound on the curvature, we obtain a uniform $W^2_2$-bound for $\tilde{f}_l$, which denotes the reparametrized and translated solution at time $t_l$. This motivates us to consider solutions in the space $W^1_{2} ((0, T_{max}); L_{2} (I; \mathbb{R}^2)) \cap L_{2} ((0, T_{max}); W^{4}_{2} (I; \mathbb{R}^2))$, as its temporal trace space is $W^{2}_{2} (I; \mathbb{R}^2))$. The idea is to restart the flow at these times $t_l$ in order to extend the flow beyond $T_{max}$, since this contradicts the maximality of the solution. In order to achieve this, we have to find a uniform lower bound on the existence time of the solutions by a compactness argument. To this end, we work with initial data of class $W^{\beta}_{2} (I; \mathbb{R}^2))$, $\beta < 2$. Finally, this enables us to extend the original solution and provides a contradiction to the maximality of the solution.

More precisely, we exploit a short time existence result Theorem \ref{local} in \cite{butzprint1}, which guarantees that the curve diffusion flow starts for initial curves given by a certain sufficiently small height function of class $W^{4(\mu - \nicefrac{1}{2})}_2(I)$, $\mu \in (\tfrac{7}{8}, 1],$ of a reference curve in $C^5(\bar{I}; \mathbb{R}^2)$. For the blow-up criterion, it will be crucial to make sure that it starts for every admissible $W_2^{4\left(\mu - \nicefrac{1}{2}\right)}(I; \mathbb{R}^2)$-curve $f_0$. 
The second main result reads:
\begin{satz} [Local Well-Posedness for a Fixed Initial Curve] \label{localbetter}
	Let $f_0: \bar{I} \to \mathbb{R}^2$, $I := (0, 1)$, be a regular curve in $W_2^{4\left(\mu - \nicefrac{1}{2}\right)}(I; \mathbb{R}^2)$, $\mu \in \left(\frac{7}{8}, 1 \right]$. Furthermore, let it fulfill the boundary conditions \eqref{initialdatum}
	Then, there exists a $T > 0$, such that $f$
	is a strong solution to curve diffusion flow.
\end{satz}

In order to prove the result, we have to assure that for every admissible regular initial curve there exists a suitable reference curve. By evolving the initial curve by a linear parabolic equation, we obtain a smoothened curve close to the initial curve. In the following, we use $C_0$-semigroup and interpolation theory to carry out technical estimates, which provide control on the distance of the two curves. Moreover, we find conditions on the distance of two curves which guarantee that one curve is a reference curve of the other one. Combining those steps enables us to confirm that the solution of the aforementioned parabolic equation is in fact a suitable reference curve.\\

A related result for the curve diffusion flow with a free boundary was proven by Wheeler and Wheeler in \cite{wheelerwheeler}: They consider immersed curves supported on two parallel lines moving according to curve diffusion flow, such that the evolving curves are orthogonal to the boundary and satisfy a no flux condition. A blow-up criterion is given in terms of the sum of the position vector and the $L_2$-norm of the arc length derivative of the curvature. Moreover, they establish criteria for global existence of the flow. \\

This article is organized as follows: In Section \ref{pre}, we present some preliminary results. Then, we construct reference curves to general admissible initial data in Section \ref{const}. Finally, we give a proof of the blow-up criterion Theorem \ref{c} in Chapter \ref{main}.

\section{Preliminaries} \label{pre} 

In this section, we introduce weighted Slobodetskii space and some of their properties. Moreover, we present two inequalities for the length and a local wellposedness result for the curve diffusion flow.

\subsection{Weighted Sobolev-Slobodetskii Spaces and Some Properties} 

In the following, we present some results on weighted Sobolev-Slobodetskii spaces. For more details, the reader is referred to \cite{meydiss} and \protect{\cite{meyries_inter}}.

\begin{definition} [Weighted Lebesgue Space]
	\label{defleb}
	Let $J = (0, T)$, $0 < T \leq \infty$ and $E$ be a Banach space. For $1 < p < \infty$ and $\mu \in \left( \frac{1}{p}, 1 \right]$ the \textbf{weighted Lebesgue space} is given by
	\begin{align*}
	L_{p, \mu}(J; E) := \left\{u: J \to E \textrm{ is strongly measurable}: \|u\|_{L_{p, \mu}(J; E)} <  \infty \right\},
	\end{align*} 
	where
	\begin{align*}
	\|u\|_{L_{p, \mu}(J; E)} := \left\| \left[t \mapsto t^{1- \mu} u(t) \right] \right\|_{L_{p}(J; E)} = \left( \int_{J} t^{(1-\mu)p} \|u(t)\|_E^{p} \dd t \right)^{\frac{1}{p}}.
	\end{align*}
\end{definition}

\begin{rem} \label{nummer}
	\begin{enumerate}
		\item $(L_{p, \mu}(J; E), \|u\|_{L_{p, \mu}(J; E)})$ is a Banach space.
		\item One easily sees that for $T < \infty$ it follows
		\begin{align*}
		L_{p}(J; E) \hookrightarrow L_{p, \mu}(J; E).
		\end{align*}
		This does not hold true for $T = \infty$.
		\item We have $L_{p, \mu}((0, T); E) \subset L_{p}((\tau, T); E)$ for $\tau \in (0, T)$.
		\item For $\mu = 1$ it holds $L_{p,1}(J; E) = L_{p}(J; E)$.
	\end{enumerate}
\end{rem}

Moreover, we define associated weighted Sobolev spaces.
\begin{definition} [Weighted Sobolev Space]
	\label{defsobo}
	Let $J = (0, T)$, $0 < T \leq \infty$ and $E$ be a Banach space. For $1 \leq p < \infty$, $k \in \mathbb{N}_0$, and $\mu \in \left( \frac{1}{p}, 1 \right]$ the \textbf{weighted Sobolev space} is given by
	\begin{align*}
	W^k_{p, \mu}(J; E) 
	:= \left\{u \in W^k_{1, \textrm{loc}} (J; E): u^{(j)} \in L_{p, \mu}(J; E) \textrm{ for } j\in\{0, \dots, k\}  \right\}
	\end{align*}
	for $k \neq 0$, where $u^{(j)} := \big(\frac{\dd}{\dd t} \big)^j u$, and we set $W^0_{p, \mu}(J; E) := L_{p, \mu}(J; E)$. We equip it with the norm
	\begin{align*}
	\|u\|_{W^k_{p, \mu}(J; E)} := \left(\sum_{j=0}^{k} \left\|u^{(j)} \right\|^p_{L_{p, \mu}(J; E)} \right)^{\frac{1}{p}}.
	\end{align*}
\end{definition}

\begin{rem}
	$(W^k_{p, \mu}(J; E), \|u\|_{W^k_{p, \mu}(J; E)})$ is a Banach space, see Section 3.2.2 of \protect{\cite{triebel}}.
\end{rem}

In the following, we introduce a generalization of the usual Sobolev spaces by means of interpolation theory. By $(\cdot,\cdot)_{\theta, p}$ we denote real interpolation functor, cf.\ \protect{\cite{lunardi}} or \protect{\cite{triebel}} for general properties of real interpolation.
\begin{definition} [Weighted Sobolev-Slobodetskii Space]
	\label{defslobo} \label{defbessel}
	Let $J = (0, T)$, $0 < T \leq \infty$ and $E$ be a Banach space. For $1 \leq p < \infty$, $s \in \mathbb{R}^+ \backslash \mathbb{N}$, and $\mu \in \left( \frac{1}{p}, 1 \right]$ the \textbf{weighted Sobolev-Slobodetskii space} is given by
	\begin{align*}
	W^s_{p, \mu} (J; E) &:= \left(W^{\lfloor s \rfloor}_{p, \mu}(J; E), W^{\lfloor s \rfloor + 1}_{p, \mu}(J; E)\right)_{s - \lfloor s \rfloor, p}.
	\end{align*}
\end{definition}

\begin{rem} \label{nummer1}
	\begin{enumerate} 
		\item $W^s_{p, \mu} (J; E)$ is a Banach space by interpolation theory, cf.\ Proposition 1.2.4 in \protect{\cite{lunardi}}.
		\item We have $W^s_{p, 1}(J; E) =  W^s_{p}(J; E)$ for all $s \geq 0$.
		\item By equation (2.8) in \protect{\cite{meyries_inter}}, we have for $s = \lfloor s \rfloor + s^*$
		\begin{align*}
		W^s_{p, \mu}(J; E) &= \left\{u \in W^{\lfloor s \rfloor}_{p, \mu}(J; E) : u^{(\lfloor s \rfloor)} \in W^{s^*}_{p, \mu}(J; E) \right\},
		\end{align*}
		where the natural norm is equivalent to the interpolation space norm with constants independent of $J$.
		\item By interpolation theory, see (2.6) in \protect{\cite{meyries_inter}}, we have the following representation of the Sobolev-Slobodetskii space: For $s \in (0, 1)$ it holds
		\begin{align*}
		W^s_{p, \mu} (J; E) =& \left\{ u \in L_{p, \mu} (J; E): [u]_{W^s_{p, \mu} (J; E)} < \infty \right\},
		\end{align*}
		where 
		\begin{align*}
		[u]_{W^s_{p, \mu} (J; E)} := \left( \int_0^T \int_0^t \tau^{(1 - \mu)p} \frac{\|u(t) - u(\tau)\|_E^p}{|t - \tau|^{1 + sp}} \dd \tau \dd t \right) ^{\frac{1}{p}}.
		\end{align*}
		Then, the norm given by
		\begin{align*}
		\|u\|_{W^s_{p, \mu} (J; E)} := \|u\|_{L_{p, \mu} (J; E)} + [u]_{W^s_{p, \mu} (J; E)} 
		\end{align*}
		is equivalent to the one induced by interpolation.
		\item If $E = \mathbb{R}$ is the image space, we omit it, e.g.\ $W^s_{p, \mu} (J) := W^s_{p, \mu} (J; \mathbb{R})$.
	\end{enumerate}
\end{rem}

We will use the following results involving weighted Sobolev-Slobodetskii spaces. 

\begin{prop} 
Let $0 < T_0 < \infty$ be fixed and $J = (0, T)$ for $0 < T \leq T_0$. Moreover, let $\mu \in \left(\frac{1}{p}, 1 \right]$ and $E$ be a Banach space. 			\label{unab0b} Let $1 < p < \infty$, $2 \geq s > 1 - \mu + \frac{1}{p}$, and $\alpha \in (0, 1)$.\\ 
			Then $W^{s}_{p, \mu}(J; E) \hookrightarrow C^{\alpha}(\bar{J}; E)$ for $s - \left(1 - \mu\right) + \frac{1}{p}  > \alpha > 0$ with the estimate
			\begin{align*}
			\|\rho\|_{C^{\alpha}(\bar{J}; E)} \leq C(T_0) (\|\rho\|_{W^{s}_{p, \mu}(J; E)} + \|\rho_{| t = 0}\|_{E}),
			\end{align*}
	where $C(T_0)$ does not depend on $T\in (0,T_0]$.
\end{prop}

\begin{rem} \label{afterunab}
	By the latter proposition, we can also prove the following:
	Let $1 < p < \infty$, $k \in \mathbb{N}$. Then $W^{s}_{p}(J; E) \hookrightarrow C^{k, \alpha}(\bar{J}; E)$ for $s - \nicefrac{1}{p} > k + \alpha > 0$ with the estimate
	\begin{align*}
	\|\rho\|_{C^{k, \alpha}(\bar{J}; E)} \leq C \|\rho\|_{W^{s}_{p}(J; E)},
	\end{align*}
	where $C$ depends on $J$.
\end{rem}

The proof can be found in Lemma 2.1.15 in \protect{\cite{butzdiss}}. 
Moreover, we deduce the following lemma, cf.\ Lemma 2.1.17 in \protect{\cite{butzdiss}}.
\begin{lem} 
	Let $0 < T_0 < \infty$ be fixed and $J= (0, T)$ for $0 < T \leq T_0$. Moreover, let $\mu \in \left(\frac{7}{8},1 \right]$. 
			\label{lem3} Let $f \in W^{s}_{2, \mu} (J)$, $1 > s > \frac{1}{2} - \mu$ such that there exists a $\tilde{C} > 0$ with $|f| \geq \tilde{C}$. Then $\tfrac{1}{f} \in W^{s}_{2, \mu} (J)$ with
			\begin{align*}
			\left\|\tfrac{1}{f} \right\|_{W^{s}_{2, \mu} (J)} \leq C \left(\|f\|_{W^{s}_{2, \mu} (J)}, \tilde{C}, T_0 \right).
			\end{align*}
\end{lem}

Additionally, we will exploit the following embedding, cf.\ Lemma 2.2.3 in \cite{butzdiss}.
\begin{lem}
	Let $T_0$ be fixed, $J = (0, T)$, $0 < T \leq T_0$, and $I$ a bounded open interval. Let $E$ be a Banach space.
	\label{unab1}
	Then
	\begin{align*}
	W^1_{2, \mu} \big(J; L_{2} (I; E)\big) \cap L_{2, \mu} \big(J; W^{2m}_{2} (I; E)\big)  \hookrightarrow BUC \big(\bar{J}, W_2^{2m\left(\mu - \nicefrac{1}{2}\right)} (I; E) \big)
	\end{align*}
	with the estimate
	\begin{align*}
	\|\rho\|_{BUC (\bar{J}, X_{\mu, E})} \leq C(T_0) \left(\|\rho\|_{W^1_{2, \mu} (J; L_{2} (I; E)) \cap L_{2, \mu} (J; W^{2m}_{2} (I; E))} + \|\rho_{| t = 0}\|_{W_2^{2m (\mu - \nicefrac{1}{2})} (I; E) } \right). 
	\end{align*}
\end{lem}

\subsection{Two Estimates for the Length }

We introduce the following quantities:
\begin{definition} [Length, Energy]
	Let $g: \bar{I} \rightarrow \mathbb{R}^2$, $x \mapsto g(x)$, be a regular $C^1$-curve. We define the \textbf{length of the curve $g$} and the \textbf{energy of the curve $g$} by
\begin{align*}
	\mathcal{L}[g] :=  \int_{I} |\partial_x g(x)| \dd x, &\qquad
	\mathcal{E}[g] := \mathcal{L}[g] + \cos{\alpha} [g(0) - g(1)]_1, \ \text{resp.}
\end{align*}
\end{definition}

\begin{lem} \label{kappabound}
	Let $\alpha \in (0, \pi)$. Furthermore, let $c: [0, 1] \rightarrow \mathbb{R}^2, \sigma \mapsto c(\sigma),$ be a regular curve of class $C^2$ parametrized proportional to arc length. Moreover, let the unit tangent $\tau := \frac{\partial_\sigma c}{\mathcal{L}[c]}$ fulfill $\tau (\sigma) = (\cos \alpha, \pm \sin \alpha)^T$ for $\sigma = 0, 1.$
	Then it holds
	\begin{align*}
	\frac{1}{\mathcal{L}[c]} \leq \frac{1}{\sqrt{2}\sin \alpha} \left\|\kappa[c] \right\|_{C([0, 1])}.
	\end{align*}
\end{lem}

The proof can be found in Lemma 2.14 in \cite{butzprint1} or Lemma 2.3.1 in \cite{butzdiss}.\\ 

\begin{lem} \label{lengthbounded}
Let $f$ be a strong solution of the curve diffusion flow. Then it holds
for $\alpha \in (0, \pi)$ and $\tilde{t} \in [0, T)$
\begin{align*}
\mathcal{L}[f(t)] \leq \frac{\mathcal{E}[f(\tilde{t})]}{1 - |\cos \alpha|} && \textrm{ for all } t \in [\tilde{t}, T).
\end{align*}
\end{lem}
The proof can be found in Remark 3.2.5 in \cite{butzdiss}.

\subsection{A Local Wellposedness Result for the Curve Diffusion Flow}

Let $\Phi^*: [0, 1] \rightarrow \mathbb{R}^2$ be a regular $C^5$-curve parametrized proportional to arc length, such that $\Lambda := \Phi^* ([0,1])$ fulfill the conditions
\begin{align}
\Phi^*(\sigma) &\in \mathbb{R} \times  \{0\} && \text{ for } \sigma \in \{0, 1\}, \nonumber \\
\measuredangle \left({n}_{\Lambda}(\sigma),  \begin{pmatrix} 0 \\ -1\end{pmatrix} \right) &= \pi - \alpha && \text{ for } \sigma \in \{0, 1\}, \nonumber \\
{\kappa}_{\Lambda} (\sigma) &= 0 && \text{ for } \sigma \in \{0, 1\},
\label{bound1}
\end{align}
where $\tau_{\Lambda}(\sigma) := \nicefrac{\partial_\sigma \Phi^* (\sigma)}{\mathcal{L}[\Phi^* ]}$
and $n_{\Lambda}(\sigma) := R\tau_{\Lambda}(\sigma)$ are the unit tangent and unit normal vector of $\Lambda$ at the point $\Phi^*(\sigma)$ for $\sigma \in [0, 1]$, respectively. Here, $R$ denotes the counterclockwise $\nicefrac{\pi}{2}$-rotation matrix. Furthermore, the curvature vector of $\Lambda$ at $\Phi^*(\sigma)$ is given by $\vec{\kappa}_{\Lambda}(\sigma) := \nicefrac{\partial^2_\sigma \Phi^* (\sigma)}{(\mathcal{L}[\Phi^* ])^2}$ for $\sigma \in [0, 1]$.

For a sufficiently small $d$ we define curvilinear coordinates by
\begin{align}
\Psi: [0, 1] \times (- d, d) & \rightarrow \mathbb{R}^2 \nonumber \\
(\sigma, q) & \mapsto \Phi^*(\sigma) + q \left(n_{\Lambda}(\sigma) + {\cot{\alpha}} \eta(\sigma) \tau_{\Lambda} (\sigma) \right), \label{curvilin}
\end{align}
where the function $\eta: [0, 1] \rightarrow [-1, 1]$ is given by
\begin{align}
\eta(x) := \begin{cases}
-1 & \textrm{ for } 0\leq x < \frac{1}{6}\\
0 & \textrm{ for } \frac{2}{6} \leq x < \frac{4}{6}\\
1 & \textrm{ for } \frac{5}{6} \leq x \leq 1 \\
\textrm{arbitrary} & \textrm{ else},
\end{cases}
\label{eta}
\end{align}
such that it is monotonically increasing and smooth. Note, that if $\alpha = \nicefrac{\pi}{2}$, then $\cot \alpha =0$ and the second summand in the definition of $\Psi$ vanishes, cf.\ \eqref{curvilin}.

In the following, $\rho: [0, T] \times [0, 1] \rightarrow (-d, d)$ serves as a height function such that $\Gamma$ can be expressed via \eqref{curvilin}. The main result of \cite{butzprint1} reads:
\begin{satz} [Local Well-Posedness for Data Close to a Reference Curve] \label{local}
	Let $\Phi^*$ and $\eta$ be given such that \eqref{bound1} and \eqref{eta} are fulfilled, respectively. Let $\rho_0 \in W_2^{4\left(\mu - \nicefrac{1}{2}\right)} (I)$, with $I = (0, 1)$ and $\mu \in \left(\frac78, 1 \right]$, fulfill the conditions
	\small
	\begin{align}
	\|\rho_0\|_{C(\bar{I})} &< \frac{K_0}{3} &&\textrm{ with }
	K_0(\alpha, \Phi^*, \eta) :=
	\frac{1}{2 \|\kappa_{\Lambda}\|_{C(\bar{I})} \left( 1 + (\cot \alpha)^2 + \hat{C}|\cot{\alpha}| \|\eta' \|_{C([0, 1])}  \right)} \\
	\|\partial_\sigma \rho_0\|_{C(\bar{I})} &< \frac{K_1}{3} &&\textrm{ with } K_1(\alpha, \Phi^*) := \frac{\mathcal{L}[\Phi^*]}{12 |\cot{\alpha}|},
	\label{small}
	\end{align}	
	\normalsize
	where $\hat{C} := \sqrt{2}\sin \alpha > 0$
	and the compatibility condition
	\begin{align}
	\partial_{\sigma} \rho_0 (\sigma) = 0 && \textrm{ for } \sigma \in \{0, 1\}.
	\label{komp}
	\end{align}
	 Furthermore, let $T_0>0$ and $R_2>0$ such that $\|\mathcal{L}^{-1} \| \leq R_2$, where $\mathcal{L}^{-1}$ is the solution operator for a linear partial differential equation depending on $\Phi^*$ and $\rho_0$.
	Then there exists a $T=T(\alpha, \Phi^*, \eta, R_1, R_2) > 0$, $\|\rho_0\|_{X_{\mu}} \leq R_1$, such that the quasilinear parabolic partial differential equation for the height function representing $\Gamma$, cf.\ (3.18) \cite{butzprint1}, possesses a unique solution $\rho \in W^1_{2, \mu} ((0, T); L_{2} (I)) \cap L_{2, \mu} ((0, T); W^{4}_{2} (I))$, such that $\rho(t)$ is a regular parametrization for all $t \in [0, T]$, and $\rho (\cdot, 0) = \rho_0$ in $ W_2^{4\left(\mu - \nicefrac{1}{2}\right)} (I)$.
\end{satz}

\begin{rem}
If follows directly that $[0,T) \times [0, 1] \ni (t, \sigma) \mapsto \Psi(\sigma, \rho(t, \sigma))$ is a strong solution of the curve diffusion flow with $\Phi(0, \cdot) = \Phi(\rho_0)$.
\end{rem}

\section{Construction of Reference Curves \label{const}} 

The idea of the proof is inspired by the argumentation for 
the approximation of $C^2$-hypersurfaces in Subsection 2.3 in \cite{pruess}.

\subsection{Generation of Potential Reference Curves}  \label{A} 

Let $f_0: \bar{I} \rightarrow \mathbb{R}^2$, $I:= (0, 1)$, be parametrized proportional to arc length and be in 
$W^{4(\mu - \nicefrac{1}{2})}_2(I; \mathbb{R}^2)$, $\mu \in \left(\nicefrac{7}{8}, 1 \right]$. Moreover, let $f_0$ fulfill the boundary conditions in \eqref{initialdatum}.

In order to apply the short time existence result, Theorem \ref{local}, to the curve $f_0$, we have to construct a suitable reference curve for $f_0$. More precisely, we look for a regular function $\Phi^*: [0, 1] \rightarrow \mathbb{R}^2$ of class $C^5$, such the boundary conditions in \eqref{bound1} are fulfilled.
We use the same notations as before. Later on, we will take care of the smallness condition of the corresponding initial height function, see \eqref{small} in Theorem \ref{local}.\\

We find a family of such curves by solving a parabolic equation subject to the previously mentioned boundary conditions \eqref{bound1} with the initial value $f_0$. The required regularity will be proven by the regularizing effects of parabolic equations. We consider
\begin{align}
\partial_t f - \partial_x^6 f &= 0 & \textrm{ for } & x \in (0, 1) \textrm{ and } t > 0, \nonumber \\
f &= f_0 & \textrm{ for } & x \in \{0, 1\}\textrm{ and } t > 0, \nonumber \\
\partial_x f &= \partial_x f_0 & \textrm{ for } & x \in \{0, 1\}\textrm{ and } t > 0, \nonumber \\
\partial_x^2 f &= 0 & \textrm{ for } & x \in \{0, 1\}\textrm{ and } t > 0, \nonumber \\
f_{|t=0} &= f_0 & \textrm{ for } & x \in (0, 1), \label{sys}
\end{align}
where the first order condition is a representation of the $\alpha$-angle condition and the second order one implies that $\vec{\kappa}[f(t)] = 0$ at the boundary for $t > 0$.

\begin{rem}
	In general dimension, it is not possible to use $f_0$ as boundary data by reasons of the regularity of the function. It works in this case, since the boundary consists only of two points, which means that the spatial regularity can be neglected.
\end{rem}

The following lemma provides a well-posedness result for \eqref{sys} with initial data $f_0$ of class
$W_2^{4 (\mu - \nicefrac{1}{2})}$. In order to solve a linear problem with optimal regularity, we will use Theorem 2.1 in \protect{\cite{meyries}} on maximal $L_p$-regularity with temporal weights. To this end, we use the notation
\begin{align*}
	X_{\tilde{\mu}, \mathbb{R}^2} &= W_2^{6\left(\tilde{\mu} - \nicefrac{1}{2}\right)} (I; \mathbb{R}^2) = W_2^{4\left(\mu - \nicefrac{1}{2}\right)} (I; \mathbb{R}^2) && \textrm{ for } \tilde{\mu}(\mu) = \tfrac{2}{3} \mu + \tfrac{1}{6} \in \left(\tfrac{3}{4}, \tfrac{5}{6}\right], \nonumber \\
	\mathbb{E}_{\tilde{\mu}, T, \mathbb{R}^2} &= W^1_{2, \tilde{\mu}} \big(J; L_{2} (I; \mathbb{R}^2)\big) \cap L_{2, \tilde{\mu}} \big(J; W^{6}_{2} (I; \mathbb{R}^2)\big) , \nonumber \\
	\mathbb{E}_{0, \tilde{\mu}, \mathbb{R}^2} &= L_{2, \tilde{\mu}} \big(J; L_{2} (I; \mathbb{R}^2)\big).
	\end{align*}
All the spaces are equipped with their natural norms. 

\begin{lem} \label{lemrefcurve}
	Let $J=(0, T)$ be finite and let $f_0: \bar{I} \rightarrow \mathbb{R}^2$, $I := (0, 1)$, be parametrized proportional to arc length and in $X_{\tilde{\mu}, \mathbb{R}^2}$ for $\mu \in \left(\frac78, 1 \right]$. Then the problem \eqref{sys} possesses a unique solution $f \in \mathbb{E}_{\tilde{\mu}, T, \mathbb{R}^2}$ with $\tilde{\mu}(\mu) = \frac{2}{3} \mu + \frac{1}{6} \in \left(\frac{3}{4}, \frac{5}{6}\right]$.
\end{lem}

The claim follows by applying a maximal regularity result, Theorem 2.1 in \protect{\cite{meyries}}, for $m = 3$ and initial data of class $W_2^{6\left(\tilde{\mu} - \nicefrac{1}{2}\right)} (I; \mathbb{R}^2)$ with $\tilde{\mu} \in \left(\tfrac{3}{4}, \tfrac{5}{6}\right]$, cf.\ Lemma 6.1.2 in \protect{\cite{butzdiss}} for details. \\

Thus, the solution $f$ satisfies the desired boundary conditions \eqref{bound1} and is arbitrarily $W_2^{4 \left(\mu - \nicefrac{1}{2}\right)}(I; \mathbb{R}^2)$-close to the initial datum $f_0$, which follows by Lemma \ref{unab1}. It remains to show that $f(t,  \cdot)$, $t>0$, is smooth enough to use it as a reference curve.
\begin{lem} \label{regular}
	Let $f$ be the solution of problem \eqref{sys} given by Lemma \ref{lemrefcurve}.
	Then $f(t, \cdot) \in C^5(\bar{I}; \mathbb{R}^2)$ for all $t\in (0, T]$.
\end{lem}

\begin{proof}
	The proof is split into the following parts: We find a homogeneous problem equivalent to \eqref{sys} and use the regularization effects, which follow by the fact that the corresponding operator generates an analytic semigroup. Then we show that the regularity transfers to the solution of the original problem. \\
	
	First, we choose a function $\xi \in C^{\infty}(\bar{I}; \mathbb{R}^2)$ fulfilling
	\begin{align}
	\label{xi}
	\begin{rcases}
	\xi(x) &= f_0(x) \\
	\partial_x \xi(x) &= \partial_x f_0(x) \\
	\partial^2_x \xi(x) &= 0
	\end{rcases} && \textrm{ for } x \in \{0, 1\}
	\end{align}
	and set $u_0 := f_0 - \xi$ and $h := \partial_x^6 \xi$. If $f \in \mathbb{E}_{\tilde{\mu}, T, \mathbb{R}^2}$ is the unique solution of \eqref{sys}, then the function 
$v:= f - \xi \in \mathbb{E}_{\tilde{\mu}, T, \mathbb{R}^2}$ solves
	\begin{align}
	\partial_t v - \partial_x^6 v &= \partial_x^6 \xi & \textrm{ for } & x \in (0, 1) \textrm{ and } t > 0, \nonumber \\
	v= \partial_x v =\partial_x^2 v &= 0 & \textrm{ for } & x \in \{0, 1\}\textrm{ and } t > 0, \nonumber \\
	v_{|t=0} &= u_0 & \textrm{ for } & x \in (0, 1), \label{sys1}
	\end{align}
	and problem \eqref{sys1} can be written as the abstract Cauchy problem 
	\begin{align}
	\partial_t u(t) + A u(t) &= h(t) & \textrm{ for } t \in (0, T), \nonumber \\
	u_{|t=0} &= u_0, \label{CP}
	\end{align}
	with $A = - \partial_x^6: D(A) \rightarrow X$, $D(A) := \{u \in W_2^6(I; \mathbb{R}^2) | \; u_{|\partial I} = 0, \partial_x  u_{|\partial I} = 0, \partial_x^2 u_{|\partial I} = 0 \}$ and $X := L_2(I; \mathbb{R}^2)$. It is a well-known result, that $-A$ is the generator of an analytic $C_0$-semigroup, see Claim 6.1.5 in \cite{butzdiss} for the proof. Now we can show higher regularity for the solution of the homogeneous problem.
	\begin{claim} \label{mildsol}
		$v \in C^{\infty}\left((0, T]; C^5(\bar{I}; \mathbb{R}^2)\right)$.
	\end{claim}
	
	\begin{proof} [Proof of the claim:]
		For an $h \in X$, which is an element of $L_1 (J; X) \cap C(\bar{J}; X)$ for $J$ bounded, and a $u_0 \in X$, we consider the mild solution of \eqref{CP}, cf.\ Definition 4.1.4 in \cite{lunardi}, given by
		\begin{align*}
		u(t) = e^{-tA}u_0 + \int_0^{t} e^{-(t-s)A} h(s) \dd s = e^{-tA}u_0 + \int_0^{t} e^{-sA} h \dd s,
		\end{align*}
		where we used that $h$ does not depend on time.
		Using the basic properties of the analytic semigroup generated by $-A$, see Chapter 2.1 in \protect{\cite{lunardi}}, we conclude that for an $x \in X$ the mapping $[t \mapsto e^{-tA}x] \in C^{\infty}((0, T], D(A))$ for $T < \infty$. Combining this with the theorem on parameter integrals for separable Banach spaces, see Theorem 3.17 in \protect{\cite{escher}}, it holds
		\begin{align*}
		u \in C^{\infty}\big((0, T]; D(A)\big) \hookrightarrow C^{\infty}\left((0, T]; C^5\big(\bar{I}; \mathbb{R}^2\big)\right).
		\end{align*}
		This proves the claim. 
	\end{proof}
	
	It follows by direct calculations that $\tilde{f} := v + \xi \in \mathbb{E}_{\tilde{\mu}, T, \mathbb{R}^2}$ solves the original problem \eqref{sys} and that the function is in $C^{\infty}((0, T]; C^5(\bar{I}; \mathbb{R}^2))$. Due to the fact that the solution is unique, we obtain $f(t) \in C^5(\bar{I}; \mathbb{R}^2)$ for $t\in (0, T]$.
\end{proof}

\subsection{Characterization of Reference Curves \label{param}}

This section aims to establish certain conditions which enable us to prove that the smoothed versions of the initial curve $f_{\epsilon} := f(\epsilon, \cdot )$,  $\epsilon>0$, see Chapter \ref{A}, are suitable to use as reference curves in the short time existence result, Theorem \ref{local}. A good starting point is the formulation of conditions for the admissible initial curves for a fixed reference curve denoted by $\Phi^*$. Again, we use the curvilinear coordinates given in \eqref{curvilin}. 
We begin by specifying the requested properties of a reference curve.
\begin{definition} [Reference Curve $\Phi^*$ for the Initial Curve $f_0$] \label{refcurvedef}
	Let $\alpha \in (0, \pi)$ be fixed. Furthermore, let $\Phi^* : [0, 1] \rightarrow \mathbb{R}^2$ be a regular $C^5$-curve parametrized proportional to arc length, fulfilling the boundary conditions \eqref{bound1} and let $f_0: [0, 1] \rightarrow \mathbb{R}^2$ be a regular $W^{\beta}_2$-curve, $\beta \in \left(\frac{3}{2}, 2\right]$, fulfilling the boundary conditions \eqref{initialdatum}.
	We call \textbf{$\Phi^*$ a reference curve for the initial curve $f_0$}, if the following conditions hold true:
	\begin{enumerate}
		\item \label{condi1} There exists a regular $C^1$-reparametrization $\varphi: [0, 1] \rightarrow [0,1]$ and a function $\rho: [0, 1] \rightarrow (-d, d)$ of class $W^{\beta}_2$, such that 
		\begin{align}
		f_0(\varphi(\sigma)) = \Phi^*(\sigma) + \rho(\sigma) (n_{\Lambda}(\sigma) + {\cot{\alpha}} \eta(\sigma) \tau_{\Lambda} (\sigma)) & \textrm{ for all } \sigma \in [0, 1].
		\label{para}
		\end{align}
		\item \label{condi2} The function $\rho$ satisfies the bounds \eqref{small}, i.e.\
		\begin{align*}
		\|\rho \|_{C([0, 1])} &< \frac{1}{6 \|\kappa_{\Lambda}\|_{C(\bar{I})} \left( 1 + (\cot \alpha)^2 + \hat{C}|\cot{\alpha}| \|\eta' \|_{C([0, 1])}  \right)} = \frac{K_0(\alpha, \Phi^*)}{3}, \\
		\intertext{and in the case $\alpha \neq \frac{\pi}{2}$ additionally}
		\|\partial_\sigma \rho \|_{C([0, 1])} &< \frac{\mathcal{L}[\Phi^*]}{36 |\cot{\alpha}|} = \frac{K_1(\alpha, \Phi^*)}{3}.
		\end{align*}
		\end{enumerate}
\end{definition}

This enables us to state the main result of this section.
\begin{satz}\label{refcurve}
	Let $\alpha \in (0, \pi)$ be fixed. Furthermore, let $\Phi^* : [0, 1] \rightarrow \mathbb{R}^2$ be a regular $C^5$-curve parametrized proportional to arc length and let it fulfill the boundary conditions \eqref{bound1}. Let $f_0: [0, 1] \rightarrow \mathbb{R}^2$ be a regular $W^{\beta}_2$-curve, $\beta \in \left(\frac{3}{2}, 2\right]$, fulfilling the boundary conditions \eqref{initialdatum}.
	Moreover, let $\lambda \in (0,1)$ be given such that the conditions
	\footnotesize
	\begin{align}
	\hspace{-0,5cm}\lambda C_{\alpha} \sqrt{1 + (\cot \alpha)^2} & < \min \left\{\sin \left(\tfrac{\sqrt{(\cot \alpha)^2 + 1} - |\cot \alpha|}{4} \right),\;
	\left\{\begin{array}{c}
	\begin{aligned}
	&\sin \left( \tfrac{1}{2(4 \cdot 144)^2 (\cot \alpha)^2} \right) & \textrm{ for } \alpha \neq \tfrac{\pi}{2} \\
	& 1 &\textrm{ for } \alpha = \tfrac{\pi}{2} 
	\end{aligned}
	\end{array} \right\} \right\} \label{lam} \\
	\lambda & < \min \left\{\frac{1}{6 \sqrt{1 + (\cot \alpha)^2}}, \;
	\left\{\begin{array}{c}
	\left.\begin{aligned}
	& \tfrac{1}{144 |\cot \alpha|} &\textrm{ for } \alpha \neq \tfrac{\pi}{2} \\
	& 1 &\textrm{ for } \alpha = \tfrac{\pi}{2} 
	\end{aligned}\right\}
	\end{array}\right. \hspace{-0,25cm} \right\}
	\label{lambdasmall}
	\end{align}
	\normalsize
	hold true.
	If $f_0 \in B^{C^0}_{\xi_0}(\Phi^*)$ and $\partial_{\sigma} f_0 \in B^{C^0}_{\xi_1}(\partial_{\sigma} \Phi^*)$, for 
	\begin{align*}
	\xi_0 &= \min \left\{\overline{C_{\alpha}(\lambda)}, \frac{(\sin \alpha)^2}{2} \right\} \frac{1}{\| \kappa_{\Lambda} \|_{C([0, 1])}}, \\
	\xi_1 &= \min \left\{\frac{\sqrt{(\cot \alpha)^2 +1} - |\cot \alpha|}{4}, 
	\left\{\begin{array}{c}
	\left.\begin{aligned}
	&\tfrac{1}{2(4 \cdot 144)^2 |\cot \alpha|^2} &\textrm{ for } \alpha \neq \tfrac{\pi}{2} \\
	& 1 &\textrm{ for } \alpha = \tfrac{\pi}{2} 
	\end{aligned}\right\}
	\end{array}\right. \hspace{-0,25cm} , \frac{\sin \alpha}{2} \right\} \mathcal{L}[\Phi^*], 
	\end{align*}
	where
	\begin{align*}
	\overline{C_\alpha (\lambda)} &:= 
	\left(1 - \sqrt{(\lambda C_{\alpha} \cot \alpha)^2 + (1 - \lambda C_{\alpha})^2}\right) \in (0, \lambda]
	\end{align*}
	for $C_{\alpha}$ given by
	\begin{align*}
	C_{\alpha} := 
	{\left[1 + (\cot{\alpha})^2 +  {\hat{C} |\cot{\alpha}}| \|\eta' \|_{C([0, 1])} \right]^{-1}} \in (0, 1]
	\end{align*}
	and $\hat{C} := (\sqrt{2}\sin \alpha)^{-1} > 0$, cf.\ Lemma \ref{kappabound}.
	Then $\Phi^*$ is a reference curve for the initial curve $f_0$. Moreover, there is a constant $C\left(\alpha, \Phi^*, \eta, \|f_0\|_{W_2^{\beta}((0, 1); \mathbb{R}^2)}\right)$ such that $\|\rho\|_{W_2^{\beta}((0, 1))} \leq C$.
\end{satz}

The first step to prove this theorem, is to show that $\Psi$ is a local diffeomorphism in a suitably small neighborhood of $\Phi^*$.
\begin{lem} \label{diffeo}
	Let $\lambda \in (0, 1]$ and $d$ be given by
	\begin{align} 
	d := \frac{C_{\alpha}}{\|\kappa_{\Lambda}\|_{C([0, 1])}},
	\label{defd}
	\end{align}
	where $C_{\alpha}$ is given as in Theorem \ref{refcurve}.
	Then $\Psi$ is a local diffeomorphism on $[0, 1] \times (-\lambda d, \lambda d)$ with
	\begin{align*}
	|D \Psi| ({\sigma}, q) > (1-\lambda)\mathcal{L}[\Phi^*] && \textrm{ for } ({\sigma}, q) \in [0, 1] \times (-\lambda d, \lambda d).
	\end{align*}
\end{lem}

The proof follows easily by calculating the derivatives of $\Psi$ and using Lemma \ref{kappabound}. \\

Now, the first criterion for reference curves can be formulated.
\begin{lem} \label{refcurve1}
	Let $\alpha \in (0, \pi)$ be fixed. Furthermore, let $\Phi^* : [0, 1] \rightarrow \mathbb{R}^2$ be a regular $C^5$-curve parametrized proportional to arc length and let it fulfill the boundary conditions \eqref{bound1}. Let $f_0: [0, 1] \rightarrow \mathbb{R}^2$ be a regular $W^{\beta}_2$-curve, $\beta \in \left(\frac{3}{2}, 2\right]$, fulfilling the boundary conditions \eqref{initialdatum}.
	Moreover, let the following conditions hold true for a $\lambda \in (0, 1)$ fulfilling the assumptions \eqref{lam} and \eqref{lambdasmall}:
	\begin{enumerate}
		\item \label{(1)} The initial curve $\Gamma_0 := f_0 (\bar{I})$ is contained in $\Psi \left([0, 1] \times (- \lambda d, \lambda d)\right)$. 
		\item \label{(2)} The curves $f_0$ and $\Phi^*$ fulfill the conditions
		\footnotesize
		\begin{align*}
		|f_0 (\sigma) - \Phi^* (\sigma) | & < \frac{\overline{C_{\alpha} (\lambda)}}{\| \kappa_{\Lambda} \|_{C([0, 1])}},\\
		|\partial_{\sigma} f_0 (\sigma) - \partial_{\sigma} \Phi^* (\sigma) | & < \min \left\{\tfrac{\sqrt{(\cot \alpha)^2 +1} - |\cot \alpha|}{4} , \left\{\begin{array}{c}
		\left.\begin{aligned}
		&\tfrac{1}{2(4 \cdot 144)^2 |\cot \alpha|^2} &\textrm{ for } \alpha \neq \tfrac{\pi}{2} \\
		& 1 &\textrm{ for } \alpha = \tfrac{\pi}{2} 
		\end{aligned}\right\}
		\end{array}\right. \hspace{-0,25cm} \right\} \mathcal{L}[\Phi^*]
		\end{align*} 
	for all $\sigma \in [0, 1]$.
	\end{enumerate}
\normalsize
	Then $\Phi^*$ is a reference curve for the initial curve $f_0$.
\end{lem}

\begin{proof}
	The proof is done in two steps: In the first step, we show that there exist functions $\varphi$ and $\rho$ fulfilling Definition \ref{refcurvedef}.\ref{condi1}. Then we prove the bounds given in Definition \ref{refcurvedef}.\ref{condi2}.\\
	
	\hspace{-0.35cm}\textit{\uline{Step 1:} Finding $\varphi$ and $\rho$ fulfilling \eqref{para}\\}
	In order to prove this, we want to use the implicit function theorem. To this end, we need the local invertibility of the function $\Phi^*$ on a slightly larger open set: We extend the function $\Phi^*$ to $(- \tilde{\delta}, 1 + \tilde{\delta})$ for a small $\tilde{\delta} > 0$: We define the extension of $\Phi^*$, which we again denote by $\Phi^*$, by
	\begin{align*}
	\Phi^* (\sigma) &:= \begin{cases}
	\Phi^*(0) + \sigma \partial_{\sigma} \Phi^* (0) & \textrm{ for } \sigma \in (-\tilde{\delta}, 0), \\
	\Phi^*({\sigma}) & \textrm{ for } \sigma \in [0, 1], \\
	\Phi^*(1) + (\sigma - 1) \partial_{\sigma} \Phi^*(1) & \textrm{ for } \sigma \in (1, 1 + \tilde{\delta}).
	\end{cases}
	\end{align*}
	Furthermore, we extend the function $\eta$ by
	\begin{align*}
	\eta (\sigma) &:= \begin{cases}
	\eta (0) & \textrm{ for } \sigma \in (-\tilde{\delta}, 0), \\
	\eta (\sigma) & \textrm{ for } \sigma \in [0, 1], \\
	\eta (1) & \textrm{ for } \sigma \in (1, 1 + \tilde{\delta}), 
	\end{cases}
	\end{align*}
	and call the extension $\eta$ again. We observe that for a sufficiently small $\tilde{\delta}>0$ the function $\Psi$ is still a local diffeomorphism on $(-\tilde{\delta}, 1 + \tilde{\delta}) \times (-d, d)$, where we used Lemma \ref{diffeo}. For the following argumentation, it is convenient to decompose the local inverse of $\Psi$ at the point $\Psi( {\tilde{\sigma}}, q) = p$ into $\Psi^{-1} = (\Pi_{\Lambda}, d_{\Lambda})$, such that 
	\begin{align}
	\Pi_{\Lambda} \in C^1 \big(U; (-\tilde{\delta}, 1 + \tilde{\delta}) \big) && \textrm{ and } && d_{\Lambda} \in C^1 \big(U;  (-d, d) \big),
	\label{decomp}
	\end{align}
	where $U$ is a suitable neighborhood of $p$ in $\mathbb{R}^2$. Note that the inverse is not necessarily unique. \\ 
	
	Moreover, we need to extend the initial curve $f_0$. We denote by  $\tau_{\Gamma_0}(\sigma) = \nicefrac{\partial_\sigma f_0 (\sigma)}{|\partial_\sigma f_0 (\sigma)|}$
	and $n_{\Gamma_0}(\sigma) = R\tau_{\Gamma_0}(\sigma)$ the unit tangent vector and the unit normal vector of $\Gamma_0$ at $f_0(\sigma)$, respectively, for $\sigma \in [0, 1]$. Then we set
	\begin{align*}
	f_0 ({\tilde{\sigma}}) & := \begin{cases}
	f_0(0) + {{\tilde{\sigma}}} \tau_{\Gamma_0} (0) & \textrm{ for } {\tilde{\sigma}} \in (-\delta, 0), \\
	f_0({{\tilde{\sigma}}}) & \textrm{ for } {\tilde{\sigma}} \in [0, 1], \\
	f_0(1) + (1 - {{\tilde{\sigma}}}) \tau_{\Gamma_0} (1) & \textrm{ for } {\tilde{\sigma}} \in (1, 1 + \delta),
	\end{cases}
	\end{align*}
	and denote the extension by $f_0$ again. Hence, we can define
	\begin{align*}
	H: &(-\delta, 1 + \delta) \times (-\tilde{\delta}, 1 + \tilde{\delta}) \times (-d, d) \rightarrow \mathbb{R}^2 \\
	({\tilde{\sigma}}, \sigma, q) \mapsto &~ f_0({\tilde{\sigma}})- \Psi(\sigma, q)  = f_0({\tilde{\sigma}}) - \Phi^*(\sigma) - q (n_{\Lambda}(\sigma) + {\cot{\alpha}} \eta(\sigma) \tau_{\Lambda} (\sigma)).
	\end{align*}
	For a fixed but arbitrary ${{\tilde{\sigma}}}_0 \in [0,1]$ it follows by \eqref{decomp} that
	\begin{align*}
	H({{\tilde{\sigma}}}_0, \underbrace{\Pi_{\Lambda}(f_0({\tilde{\sigma}}_0))),d_{\Lambda}(f_0({\tilde{\sigma}}_0))}_{=: y_0}) = 0.
	\end{align*}
	Since $\Psi$ is a local diffeomorphism on $(-\tilde{\delta}, 1 + \tilde{\delta}) \times (-d, d)$, the function $\partial_y H$ is invertible in $({\tilde{\sigma}}_0, y_0)$. Moreover, $H$ and $\partial_y H$ are continuous in $({\tilde{\sigma}}_0, y_0)$. Thus, we obtain by the implicit function theorem, cf.\ Theorem 4.B in \protect{\cite{zeidler}}, that there exists a neighborhood $U_0 \subset (-\delta, 1 + \delta)$ of ${\tilde{\sigma}}_0$ and $V_0 \subset (-\tilde{\delta}, 1 + \tilde{\delta})$ of $y_0$ and a function $g: U_0 \rightarrow V_0$ such that $H({\tilde{\sigma}}, g({\tilde{\sigma}})) = 0$. As $f_0$ and $\Psi$ are $C^1$-maps, $g$ is as well of class $C^1$ in a neighborhood of $x_0$. \\
	
	In this way, we obtain for each ${\tilde{\sigma}}_0 \in [0,1]$ a function $g$, which is just locally defined in $U_0$. As $\Psi$ is a local $C^1$-diffeomorphism on the domain, $g$ is uniquely determined. Thus, two functions $g_{a}$ and $g_{b}$, for $a, b \in [0, 1]$ with $a\neq b$, have to coincide on $U_a \cap U_b$. Additionally, it follows by construction that the first component of $g$, which is denoted by $g_1$, fulfills $g_1(0) = g_1 (0) = 0$ and $g_1 (1) = 1$. By gluing together the locally defined functions, we obtain a $C^1$-function $g(\cdot)$ with $g_1 ([0,1])=[0,1]$. \\
	
	In order to define the function $\varphi$, cf.\ Definition \ref{refcurvedef}, we show that $g_1$ is injective. Consequently, it is invertible on $[0,1]$ and we can define 
	\begin{align*}
	\varphi({{\sigma}}) := g_1^{-1}(\sigma) \qquad \textrm{ and } \qquad \rho({{\sigma}}) := g_2 (\varphi({{\sigma}})) = g_2 \circ g_1^{-1}(\sigma).
	\end{align*}
	By the differential rule for inverse mappings 
        both functions are of~class~$C^1$.
	
	\begin{claim} \label{nummer4}
		The function $g_1:[0, 1] \rightarrow [0, 1]$ is injective.
	\end{claim}
	
	\begin{proof}[Proof of the claim:]
		We use the chain rule and obtain for ${\tilde{\sigma}} \in [0, 1]$
		\begin{align*}
		0 = \partial_{{\tilde{\sigma}}} (H({\tilde{\sigma}}, g({\tilde{\sigma}}))) = (\partial_{{\tilde{\sigma}}} H)({\tilde{\sigma}}, g({\tilde{\sigma}})) +
		(\partial_{y} H)({\tilde{\sigma}}, g({\tilde{\sigma}}))
		\partial_{{\tilde{\sigma}}} g({\tilde{\sigma}}).
		\end{align*}
		By Lemma \ref{diffeo}, it follows that $\partial_{{\tilde{\sigma}}} \Psi$ and $\partial_{q} \Psi$ are linearly independent for $(\sigma, q) \in [0, 1] \times (-d, d)$, where $d$ is given by \eqref{defd}. 
		Using Condition \ref{(1)}, we obtain from the previous calculations that 
		\begin{align*}
		\partial_{{\tilde{\sigma}}} g_1({\tilde{\sigma}}) = 0\qquad \Leftrightarrow \qquad
		\partial_{{{\sigma}}} f_{0}({\tilde{\sigma}}) =  \partial_{q} \Psi(g_1({\tilde{\sigma}})) \partial_{{\tilde{\sigma}}} g_2({\tilde{\sigma}}).
		\end{align*}
		Thus, we want to rule out that it holds $\tau_{\Gamma_0} (\tilde{\sigma}) \parallel  \partial_{q} \Psi(g_1({\tilde{\sigma}}))$, or equivalently $n_{\Gamma_0}  (\tilde{\sigma}) \perp \partial_{q} \Psi(g_1({\tilde{\sigma}}))$. Direct calculations provide
		\begin{align*}
		\left \langle n_{\Gamma_0}  (\tilde{\sigma}), \partial_{q} \Psi(g_1({\tilde{\sigma}})) \right \rangle &= \frac{1}{2}\left(|n_{\Gamma_0}  (\tilde{\sigma})|^2 + |\partial_{q} \Psi(g_1({\tilde{\sigma}}))|^2 - |n_{\Gamma_0}  (\tilde{\sigma}) - \partial_{q} \Psi(g_1({\tilde{\sigma}}))|^2 \right) \\
		&= \frac{1}{2} \left(1 + 1 + (\cot \alpha)^2 (\eta(g_1({\tilde{\sigma}})))^2 - |n_{\Gamma_0}  (\tilde{\sigma}) - \partial_{q} \Psi(g_1({\tilde{\sigma}}))|^2 \right).
		\end{align*}
		This implies that $\langle n_{\Gamma_0}  (\tilde{\sigma}), \partial_{q} \Psi(g_1({\tilde{\sigma}})) \rangle > \nicefrac{1}{2}$, if
		\begin{align}
		|n_{\Gamma_0}  (\tilde{\sigma}) - \partial_{q} \Psi(g_1({\tilde{\sigma}}))|^2 < 1 + (\cot \alpha)^2 (\eta({g_1({\tilde{\sigma}})}))^2
		\label{cond2}
		\end{align}
		holds true. We deduce
		\begin{align*}
		|n_{\Gamma_0}  (\tilde{\sigma}) - \partial_{q} \Psi(g_1({\tilde{\sigma}}))| &\leq |n_{\Gamma_0}  (\tilde{\sigma}) - n_{\Lambda} (g_1(\tilde{\sigma}))| + |n_{\Lambda} (g_1(\tilde{\sigma})) - \partial_{q} \Psi(g_1({\tilde{\sigma}}))| \\ 
		&\leq \underbrace{|n_{\Gamma_0}  (\tilde{\sigma}) - n_{\Lambda} (\tilde{\sigma})| + |n_{\Lambda} (\tilde{\sigma}) - n_{\Lambda} (g_1(\tilde{\sigma}))|}_{=:I + II} + |\cot \alpha \eta (g_1({\tilde{\sigma}}))| ,
		\end{align*}
		where we used the definition of $\partial_{q} \Psi$. Comparing this with \eqref{cond2}, it suffices to show that 
		\begin{align*}
		\left(I + II + |\cot \alpha \eta (g_1({\tilde{\sigma}}))| \right)^2 &< 1 + (\cot \alpha)^2 (\eta({g_1({\tilde{\sigma}})}))^2,
		\end{align*}
		which is equivalent to
		\begin{align*}
		(I + II)^2 + 2 |\cot \alpha|(I + II) &< 1.
		\end{align*}
		Solving the quadratic inequality and taking the positive solution, we deduce the condition
		\begin{align}
		I + II < {\sqrt{(\cot \alpha)^2 +1} - |\cot \alpha|} && \Rightarrow && g_1 \textrm{ is injective.}
		\label{suffcon}
		\end{align}
		
		In the following, we show that the left-hand side of \eqref{suffcon} is fulfilled by estimating $I$ and $II$ separately: For the term $I = |n_{\Gamma_0}  (\tilde{\sigma}) - n_{\Lambda} (\tilde{\sigma})|$, we obtain by adding a zero
		\begin{align}
		I & =
		\frac{\big|\partial_\sigma f_0(\tilde{\sigma}) |\partial_\sigma \Phi^* (\tilde{\sigma})| - \partial_\sigma \Phi^*  (\tilde{\sigma})|\partial_\sigma f_0(\tilde{\sigma})| \big|}{|\partial_\sigma f_0(\tilde{\sigma})||\partial_\sigma \Phi^*  (\tilde{\sigma})|} \label{nutze1} \\
		&\leq \frac{\big| |\partial_\sigma \Phi^*  (\tilde{\sigma})| - |\partial_\sigma f_0(\tilde{\sigma})| \big|}{|\partial_\sigma \Phi^*  (\tilde{\sigma})|} + \frac{|\partial_\sigma f_0(\tilde{\sigma}) - \partial_\sigma \Phi^* (\tilde{\sigma})|} {|\partial_\sigma \Phi^*  (\tilde{\sigma})|} \leq \frac{2}{|\partial_\sigma \Phi^*(\tilde{\sigma})|} |\partial_\sigma f_0(\tilde{\sigma}) - \partial_\sigma \Phi^*  (\tilde{\sigma})|, \nonumber
		\end{align}
		where ${|\partial_\sigma \Phi^*(\tilde{\sigma})|} = {\mathcal{L}[\Phi^*]}$, as $\Phi^*$ is parametrized proportional to arc length. 
		Thus, we have by Condition \ref{(2)}
		\begin{align}
		I < \frac{\sqrt{(\cot \alpha)^2 +1} - |\cot \alpha|}{2}.
		\label{one}
		\end{align}
		Considering the summand $II$, we deduce by the fundamental theorem of calculus
		\begin{align}
		II &
		\leq \left| \int_{\tilde{\sigma}}^{g_1(\tilde{\sigma})} {\mathcal{L}[\Phi^*]} \kappa_{\Lambda}(\sigma) n_{\Lambda} (\sigma)  \dd  \sigma \right| \leq
		\mathcal{L} [\Phi^* ([g_1 (\tilde{\sigma}), \tilde{\sigma} ])] \| \kappa_{\Lambda} \|_{C([0, 1])}, \label{zwerg}
		\end{align}
		since $|g_1(\tilde{\sigma}) - \tilde{\sigma}| {\mathcal{L}[\Phi^*]} = \mathcal{L} [\Phi^* ([g_1 (\tilde{\sigma}), \tilde{\sigma} ])]$ by the proportional-to-arc-length-parametrization of $\Phi^*$. Here, we assumed w.l.o.g.\ that $g_1(\tilde{\sigma}) \leq \tilde{\sigma}$. Thus, it remains to estimate $\mathcal{L} [\Phi^* ([g_1 (\tilde{\sigma}), \tilde{\sigma} ])]$ by geometric considerations. By the bound on the curvature, it follows, that $\mathcal{L} [\Phi^* ([g_1 (\tilde{\sigma}), \tilde{\sigma} ])]$ cannot be larger than the circle arc with radius $r = 1 / \| \kappa_{\Lambda} \|_{C([0, 1])}$ that connects the points $\Phi^* (g_1 (\tilde{\sigma}))$ and $\Phi^*(\tilde{\sigma})$. It is denoted by $\mathcal{L}_{max} [\Phi^* ([g_1 (\tilde{\sigma}), \tilde{\sigma} ])]$, cf.\ Figure \ref{abs}.
		\begin{center}\vspace{0.3 cm}
			\scalebox{1}{
\begingroup%
  \makeatletter%
  \providecommand\color[2][]{%
    \errmessage{(Inkscape) Color is used for the text in Inkscape, but the package 'color.sty' is not loaded}%
    \renewcommand\color[2][]{}%
  }%
  \providecommand\transparent[1]{%
    \errmessage{(Inkscape) Transparency is used (non-zero) for the text in Inkscape, but the package 'transparent.sty' is not loaded}%
    \renewcommand\transparent[1]{}%
  }%
  \providecommand\rotatebox[2]{#2}%
  \newcommand*\fsize{\dimexpr\f@size pt\relax}%
  \newcommand*\lineheight[1]{\fontsize{\fsize}{#1\fsize}\selectfont}%
  \ifx\svgwidth\undefined%
    \setlength{\unitlength}{236.07036489bp}%
    \ifx\svgscale\undefined%
      \relax%
    \else%
      \setlength{\unitlength}{\unitlength * \real{\svgscale}}%
    \fi%
  \else%
    \setlength{\unitlength}{\svgwidth}%
  \fi%
  \global\let\svgwidth\undefined%
  \global\let\svgscale\undefined%
  \makeatother%
  \begin{picture}(1,0.61132258)%
    \lineheight{1}%
    \setlength\tabcolsep{0pt}%
    \put(0,0){\includegraphics[width=\unitlength,page=1]{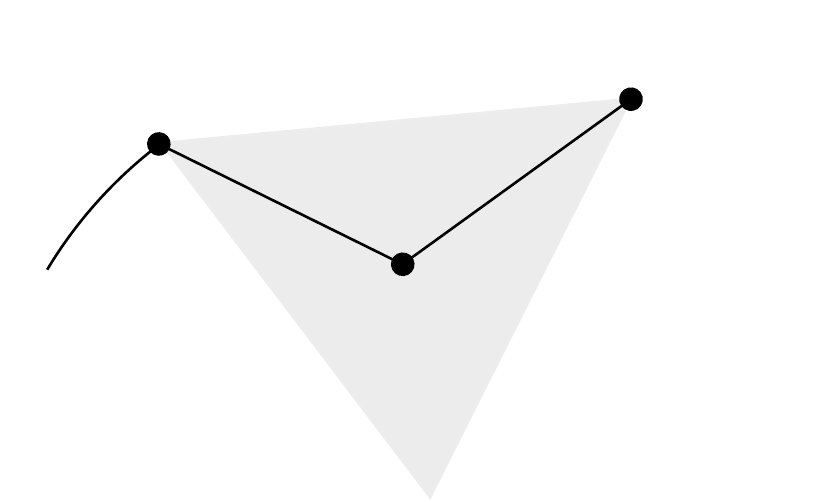}}%
    \put(0.48543392,0.33858277){\color[rgb]{0,0,0}\makebox(0,0)[t]{\lineheight{0}\smash{\begin{tabular}[t]{c}\textbf{$f_0 (\tilde{\sigma})$}\end{tabular}}}}%
    \put(0.08436245,0.46077875){\color[rgb]{0,0,0}\makebox(0,0)[t]{\lineheight{0}\smash{\begin{tabular}[t]{c}\textbf{$\Phi^* (g_1 (\tilde{\sigma}))$}\end{tabular}}}}%
    \put(0.8394606,0.50537715){\color[rgb]{0,0,0}\makebox(0,0)[t]{\lineheight{0}\smash{\begin{tabular}[t]{c}\textbf{$\Phi^* (\tilde{\sigma})$}\end{tabular}}}}%
    \put(0,0){\includegraphics[width=\unitlength,page=2]{abs1.pdf}}%
    \put(0.23665739,0.1792275){\color[rgb]{0,0,0}\rotatebox{-26.24250123}{\makebox(0,0)[t]{\lineheight{0}\smash{\begin{tabular}[t]{c}\textbf{$< \frac{\lambda  C_{\alpha}\sqrt{1 + (\cot \alpha )^2}}{\|\kappa_{\Lambda}\|_{C^0[0, 1]}} $}\end{tabular}}}}}%
    \put(0.75985862,0.2380452){\color[rgb]{0,0,0}\rotatebox{36.09328585}{\makebox(0,0)[t]{\lineheight{0}\smash{\begin{tabular}[t]{c}\textbf{$< \frac{\overline{C_{\alpha}(\lambda)}}{\|\kappa_{\Lambda}\|_{C^0[0, 1]}}$}\end{tabular}}}}}%
    \put(0.51898333,0.60449294){\color[rgb]{0,0,0}\makebox(0,0)[t]{\lineheight{0}\smash{\begin{tabular}[t]{c}\textbf{$\mathcal{L}_{max} [\Phi^* ([g_1 (\tilde{\sigma}), \tilde{\sigma} ])]$}\end{tabular}}}}%
    \put(0,0){\includegraphics[width=\unitlength,page=3]{abs1.pdf}}%
  \end{picture}%
\endgroup%
}
			\captionof{figure}{The estimate of $\mathcal{L}_{max} [\Phi^* ([g_1 (\tilde{\sigma}), \tilde{\sigma} ])]$ for $|\kappa_{\Lambda}(\sigma)| = \| \kappa_{\Lambda} \|_{C([0, 1])}$ for $\sigma \in [g_1 (\tilde{\sigma}), \tilde{\sigma} ]$.}
			\label{abs}
		\end{center}\vspace{0 cm}
		By Condition \ref{(1)} of Lemma \ref{refcurve1} it holds
		\begin{align}
		|\Phi^* (g_1 (\tilde{\sigma})) - f_0 (\tilde{\sigma})| = |d_{\Lambda}(f_0 (\tilde{\sigma}))| < \lambda d \sqrt{1 + (\cot \alpha)^2} \leq \frac{\lambda  C_{\alpha}\sqrt{1 + (\cot \alpha)^2}}{\|\kappa_{\Lambda}\|_{C^0[0, 1]}}.
		\label{nummer8}
		\end{align}
		Moreover, we have a bound on $|\Phi^* (\tilde{\sigma}) - f_0 (\tilde{\sigma})|$ for $\tilde{\sigma} \in [0, 1]$ by Condition \ref{(2)} of Lemma \ref{refcurve1}. Thus, we observe by adding a zero
		\begin{align*}
		|\Phi^* (g_1 (\tilde{\sigma})) - \Phi^* (\tilde{\sigma})| 
		&< \left[\lambda  C_{\alpha}\sqrt{1 + (\cot \alpha)^2} + \overline{C_{\alpha}(\lambda)}\right] \frac{1}{\|\kappa_{\Lambda}\|_{C^0[0, 1]}}
		\end{align*}
		and the estimated term corresponds to the length of the doted line in Figure \ref{abs}. The inequality
		\begin{align*}
		\overline{C_{\alpha}(\lambda)} = \left(1 - \sqrt{(\lambda C_{\alpha} \cot \alpha)^2 + (1 - \lambda C_{\alpha})^2}\right) < 1 - \sqrt{(1 - \lambda C_{\alpha})^2} = \lambda C_{\alpha},
		\end{align*}
		provides
		\begin{align*}
		\left|\Phi^* (g_1 (\tilde{\sigma})) - \Phi^* (\tilde{\sigma}) \right| <  \frac{2 \lambda  C_{\alpha}\sqrt{1 + (\cot \alpha)^2}}{\| \kappa_{\Lambda} \|_{C([0, 1])}}.
		\end{align*}
		Now, we can use basic geometry on the symmetric rectangular triangles, which arise by splitting the grey area in Figure \ref{abs}, to estimate $\mathcal{L} [\Phi^* ([g_1 (\tilde{\sigma}), \tilde{\sigma} ])]$.
		It follows
		\begin{align*}
		\mathcal{L}_{max} [\Phi^* ([g_1 (\tilde{\sigma}), \tilde{\sigma} ])] &\leq  2 \arcsin \left( \frac{|\Phi^* (g_1 (\tilde{\sigma})) - \Phi^* (\tilde{\sigma})| \| \kappa_{\Lambda} \|_{C([0, 1])}}{2} \right) \frac{1 }{\| \kappa_{\Lambda} \|_{C([0, 1])}}\\
		&< 2 \arcsin \left( \lambda C_{\alpha}\sqrt{1 + (\cot \alpha)^2} \right) \frac{1}{\| \kappa_{\Lambda} \|_{C([0, 1])}}. 
		\end{align*}
		By assumption \eqref{lam}, we have
		\begin{align*}
		\lambda C_{\alpha} \sqrt{1 + (\cot \alpha)^2} &< \sin \left(\frac{\sqrt{(\cot \alpha)^2 + 1} - |\cot \alpha|}{4} \right).
		\end{align*}
		Hence,
		\begin{align*}
		\mathcal{L}_{max} [\Phi^* ([g_1 (\tilde{\sigma}), \tilde{\sigma} ])] < \frac{\sqrt{(\cot \alpha)^2 + 1} - |\cot \alpha|}{2 \| \kappa_{\Lambda} \|_{C([0, 1])}}
		\end{align*}
		is deduced. By plugging this into \eqref{zwerg}, we obtain for $II = |n_{\Lambda} (\tilde{\sigma}) - n_{\Lambda} (g_1(\tilde{\sigma}))|$
		\begin{align}
		II <   2 \arcsin \left( \lambda C_{\alpha}\sqrt{1 + (\cot \alpha)^2} \right) < \frac{\sqrt{(\cot \alpha)^2 + 1} - |\cot \alpha|}{2}.
		\label{nutze2}
		\end{align}
		Combining this with \eqref{one} shows that condition \eqref{suffcon} is fulfilled.
	\end{proof}
	By construction we obtain for ${{\sigma}} \in [0,1]$
	\begin{align*}
	0 = H(\varphi({\sigma}), {{\sigma}}, \rho({{\sigma}})) = f_0(\varphi(\sigma)) - \Phi^*(\sigma) - \rho(\sigma) (n_{\Lambda}(\sigma) + {\cot{\alpha}} \eta(\sigma) \tau_{\Lambda} (\sigma)),
	\end{align*}
	hence, the identity \eqref{para} is fulfilled. Differentiation with respect to $\sigma$ yields
	\begin{align}
	0 = \partial_\sigma H(\varphi({\sigma}), {{\sigma}}, \rho({{\sigma}})) = \partial_\sigma f_0(\varphi(\sigma)) \varphi'(\sigma) - \Psi_\sigma ({{\sigma}}, \rho({{\sigma}})) - \Psi_q ({{\sigma}}) \partial_\sigma \rho({{\sigma}}).
	\label{nummer6}
	\end{align}
	
	\hspace{-0.35cm}\textit{\uline{Step 2:} Proof of the bounds given in Definition \ref{refcurvedef}.\ref{condi2}\\} 
	It remains to show that the bounds on $\rho$ and $\partial_\sigma \rho$ stated in Definition \ref{refcurvedef} hold true. To this end, the following claims are proven:
	\begin{claim} \label{hilfe}
		It holds $\|\rho \|_{C([0, 1])} < \frac{K_0(\alpha, \Phi^*, \eta)}{3}$.
	\end{claim}
	
	\begin{proof}[Proof of the claim]
		By the identity \eqref{para}, we obtain
		\begin{align*}
		\|\rho\|_{C([0, 1])} \leq \|\rho (n_{\Lambda} + {\cot{\alpha}} \eta \tau_{\Lambda})\|_{C([0, 1])} = \| f_0 \circ \varphi - \Phi^* \|_{C([0, 1])}.
		\end{align*}
		As reparametrization does not change the $\|\cdot\|_{C([0, 1])}$-norm, it follows
		\begin{align}
		\|\rho\|_{C([0, 1])} \leq \| f_0 - \Phi^* \circ g_1 \|_{C([0, 1])} < \frac{\lambda  C_{\alpha}\sqrt{1 + (\cot \alpha)^2}}{\|\kappa_{\Lambda}\|_{C^0[0, 1]}}, \label{esti}
		\end{align}
		where \eqref{nummer8} is used for the last inequality. By assumption \eqref{lambdasmall}, the inequality
		\begin{align*}
		\lambda  C_{\alpha}\sqrt{1 + (\cot \alpha)^2}  < \frac{1}{6 \left( 1 + (\cot \alpha)^2 + \hat{C}|\cot{\alpha}| \|\eta' \|_{C([0, 1])}  \right)}
		\end{align*}
		is inferred and therefore the claim.
	\end{proof}
	
	\begin{claim} \label{hilfe4}
		In the case $\alpha \neq \frac{\pi}{2}$, it holds $\left\| \partial_\sigma \rho \right\|_{C([0, 1])} < \frac{K_1(\alpha, \Phi^*)}{3}$.
	\end{claim}
	
	\begin{proof} 
		By taking the inner product of identity \eqref{nummer6} with $R\partial_\sigma f_0(\varphi(\sigma))$, where $R$ is the matrix which rotates vectors by $\nicefrac{\pi}{2}$ counterclockwise, we have
		\begin{align*}
		\langle R\partial_\sigma f_0(\varphi(\sigma)),\Psi_q ({{\sigma}}, \rho({{\sigma}})) \rangle \partial_\sigma \rho({{\sigma}}) = - \langle R\partial_\sigma f_0(\varphi(\sigma)), \Psi_\sigma ({{\sigma}}, \rho({{\sigma}}))\rangle.
		\end{align*}
		Taking into account $\langle n_{\Gamma_0}  (\tilde{\sigma}), \partial_{q} \Psi(g_1({\tilde{\sigma}})) \rangle > \nicefrac{1}{2}$, see proof of Claim \ref{nummer4}, we obtain
		\begin{align}
		\partial_\sigma \rho({{\sigma}}) = - \frac{\langle R\partial_\sigma f_0(\varphi(\sigma)), \Psi_\sigma ({{\sigma}}, \rho({{\sigma}}))\rangle}{\langle R\partial_\sigma f_0(\varphi(\sigma)),\Psi_q ({{\sigma}}, \rho({{\sigma}})) \rangle} = - \frac{\langle n_{\Gamma_0} (\varphi(\sigma)), \Psi_\sigma ({{\sigma}}, \rho({{\sigma}}))\rangle}{\langle n_{\Gamma_0} (\varphi(\sigma)),\Psi_q ({{\sigma}}, \rho({{\sigma}})) \rangle}.
		\label{nummer3}
		\end{align}
		Thus 
		\begin{align}
		\left\| \partial_\sigma \rho \right\|_{C([0, 1])}
		&\leq 2 \left\|\langle n_{\Gamma_0} (\varphi(\cdot)), \Psi_\sigma ({{\cdot}}, \rho({{\cdot}}))\rangle\right\|_{C([0, 1])}.
		\label{nummer2}
		\end{align}
		In the following, we use
		\begin{align*}
		\Psi_\sigma ({{\sigma}}, \rho({{\sigma}})) =  \mathcal{L}[\Phi^* ]\left(1 - \rho\kappa_{\Lambda} +  \rho {\cot{\alpha}} \frac{\eta'({\sigma})}{\mathcal{L}[\Phi^* ]} \right) \tau_{\Lambda}({\sigma}) + \left( \mathcal{L}[\Phi^* ] \rho {\cot{\alpha}} \eta({\sigma}) \kappa_{\Lambda} \right) n_{\Lambda}({\sigma}),
		\end{align*}
		see the calculation in the proof of Lemma \ref{diffeo}. First, we show that 
		\begin{align}
		|\langle n_{\Gamma_0} (\varphi(\sigma)), \tau_{\Lambda}({\sigma}) \rangle| = \sqrt{1 - \langle n_{\Gamma_0} (\varphi(\sigma)), n_{\Lambda}({\sigma}) \rangle^2}
		\label{name}
		\end{align}
		is small. To this end, we use again the representation
		\begin{align}
		\langle n_{\Gamma_0} (\varphi(\sigma)), n_{\Lambda}({\sigma}) \rangle = \frac{1}{2}\big(2 - | n_{\Gamma_0} (\varphi(\sigma)) - n_{\Lambda}({\sigma})|\big).
		\label{name1}
		\end{align}
		By triangle inequality, it follows
		\begin{align*}
		| n_{\Gamma_0} (\varphi(\sigma)) - n_{\Lambda}({\sigma})|\leq | n_{\Gamma_0} (\varphi(\sigma)) - n_{\Lambda}(\varphi(\sigma))| + |n_{\Lambda}(\varphi(\sigma)) - n_{\Lambda}({\sigma})| = I + II.
		\end{align*}
		We observe that we already derived bounds on both summands previously, cf.\ \eqref{nutze1} and \eqref{nutze2}. More precisely, we obtain by Condition \ref{(2)} in Lemma \ref{refcurve1} and assumption \eqref{lam} that
		\begin{align*}
		I < \frac{1}{(4 \cdot 144)^2 (\cot \alpha)^2} \qquad \textrm{ and } \qquad
		II < \frac{1}{(4 \cdot 144)^2 (\cot \alpha)^2}.
		\end{align*}
		Consequently, it holds
		\begin{align*}
		| n_{\Gamma_0} (\varphi(\sigma)) - n_{\Lambda}({\sigma})| < \frac{1}{2(4 \cdot 144)^2 (\cot \alpha)^2} =: \gamma
		\end{align*}
		and by plugging this into \eqref{name1}
		\begin{align*}
		\langle n_{\Gamma_0} (\varphi(\sigma)), n_{\Lambda}({\sigma}) \rangle > \frac{1}{2}\left(2 - 2 \gamma\right) = 1 - \gamma.
		\end{align*}
		Finally, we deduce by \eqref{name}
		\begin{align*}
		\left|\langle n_{\Gamma_0} (\varphi(\sigma)), \tau_{\Lambda}({\sigma}) \rangle \right| < \sqrt{1 - (1- \gamma)^2} = \sqrt{2\gamma - \gamma^2} < \sqrt{2\gamma} = \frac{1}{2 \cdot 144 |\cot \alpha|}.
		\end{align*}
		This yields
		\small
		\begin{align*}
		\left| \langle n_{\Gamma_0} (\varphi(\sigma)), \Psi_\sigma ({{\sigma}}, \rho({{\sigma}}))\rangle \right| < \left|\mathcal{L}[\Phi^* ]\left(1 - \rho\kappa_{\Lambda} +  \rho {\cot{\alpha}} \frac{\eta'({\sigma})}{\mathcal{L}[\Phi^* ]} \right) \right| \sqrt{2\gamma} + \left| \mathcal{L}[\Phi^* ] \rho {\cot{\alpha}} \eta({\sigma}) \kappa_{\Lambda} \right|.
		\end{align*}
		\normalsize
		Using the bound on $\rho$ in \eqref{esti}, we infer that 
		\begin{align*}
		\left| \langle n_{\Gamma_0} (\varphi(\sigma)), \Psi_\sigma ({{\sigma}}, \rho({{\sigma}}))\rangle \right| < \mathcal{L}[\Phi^*] 2 \sqrt{2\gamma} +  \mathcal{L}[\Phi^* ] \lambda = \frac{\mathcal{L}[\Phi^*]}{144 |\cot \alpha|} +  \mathcal{L}[\Phi^* ] \lambda.
		\end{align*}
		The claim follows by the assumption \eqref{lambdasmall}.
	\end{proof}
	
	Finally, the bound on $\|\rho\|_{W_2^{\beta}((0, 1))}$ is shown. 
	\begin{claim} \label{hilfe3}
		$\|\rho\|_{W_2^{\beta}((0, 1))} $ is bounded by a constant depending on $\alpha, \Phi^*, \eta$, $\|f_0\|_{W_2^{\beta}((0, 1); \mathbb{R}^2)}$.
	\end{claim}
	
	\begin{proof}[Proof of the claim]
		First, we observe that $4(\mu - \nicefrac{1}{2}) \in (\nicefrac32, 2]$ for $\mu \in (\nicefrac78, 1]$. Using the characterization of Sobolev-Slobodetskii spaces in Remark~\ref{nummer1}.(iii), we have for $s \in (1, 2)$ that 
		\begin{align*}
		u \in W^s_2(I) && \Leftrightarrow && u \in W^{1}_2(I) \wedge u' \in {W^{s^*}_2(I)} \textrm{ for } s^* := s - \lfloor s \rfloor \in (0,1).
		\end{align*}
		By the Claims \ref{hilfe} and \ref{hilfe4}, we obtain bounds on $\|\rho\|_{C([0, 1])}$ for arbitrary $\alpha \in (0, \pi)$ and on $\|\partial_{\sigma} \rho\|_{C([0, 1])}$ for $\alpha \in (0, \pi) \backslash \{\nicefrac{\pi}{2}\}$. Note that in the case $\alpha = \nicefrac{\pi}{2}$, we obtain by \eqref{nummer2} the estimate for $\left\| \partial_\sigma \rho \right\|_{C([0, 1])}$. Thus it suffices to show that
		\begin{itemize}
			\item for $4(\mu - \nicefrac{1}{2}) \in (\nicefrac32, 2)$ the semi-norm $[\partial_\sigma \rho]_{W^{s^*}_2((0, 1))}$ is bounded by a suitable constant,
			\item for $4(\mu - \nicefrac{1}{2}) = 2$ the norm $\left\|\partial^2_\sigma \rho \right\|_{L_2((0, 1))}$ is bounded by a suitable constant.
		\end{itemize}
		To this end, we use equality \eqref{nummer3}: 
		We already proved that the denominator is bounded from below, see proof of Claim \ref{nummer4}. Moreover, we observe that $W^{a}_2((0, 1)) \hookrightarrow C ([0,1])$ for $a \in (\nicefrac{1}{2}, 1]$ by Proposition 2.10 in \protect{\cite{meyries_inter}}. Then it follows 
		by Corollary 2.86 in \cite{danchin}, that $W^{s^*}_2(I)$, $s^* \in (\nicefrac{1}{2}, 1)$ is closed with respect to multiplication. Moreover, $W^{1}_2((0, 1))$ is also a Banach algebra. Additionally, by Lemma \ref{lem3}, it holds $\nicefrac{1}{f} \in W^{s^*}_2(I)$, $s^* \in (\nicefrac{1}{2}, 1)$, if $f \in W^{s^*}_2(I)$ and $f$ is bounded away from zero. The analogous result holds for $f \in W^{1}_2(I)$. The claim follows by combining these results.
	\end{proof}
	This proves Lemma \ref{refcurve1}.
\end{proof}

There is one last step to conclude Theorem \ref{refcurve} from Lemma \ref{refcurve1}: We have to show that we can substitute Condition \ref{(1)} and \ref{(2)} in Lemma \ref{refcurve1} by conditions on the $C^0$-difference of $f_0$ and $\Phi^*$, and on the one of their first derivatives. Note that Condition \ref{(2)} is already in a convenient form. We observe that Condition \ref{(1)} cannot be achieved directly by choosing the difference small enough, as it cannot rule out that the initial curve is negative in a neighborhood of the boundary points, cf.\ Figure \ref{grey}.
\begin{center}\vspace{0.3 cm} 
	\scalebox{1}{
\begingroup%
  \makeatletter%
  \providecommand\color[2][]{%
    \errmessage{(Inkscape) Color is used for the text in Inkscape, but the package 'color.sty' is not loaded}%
    \renewcommand\color[2][]{}%
  }%
  \providecommand\transparent[1]{%
    \errmessage{(Inkscape) Transparency is used (non-zero) for the text in Inkscape, but the package 'transparent.sty' is not loaded}%
    \renewcommand\transparent[1]{}%
  }%
  \providecommand\rotatebox[2]{#2}%
  \ifx\svgwidth\undefined%
    \setlength{\unitlength}{172.275bp}%
    \ifx\svgscale\undefined%
      \relax%
    \else%
      \setlength{\unitlength}{\unitlength * \real{\svgscale}}%
    \fi%
  \else%
    \setlength{\unitlength}{\svgwidth}%
  \fi%
  \global\let\svgwidth\undefined%
  \global\let\svgscale\undefined%
  \makeatother%
  \begin{picture}(1,0.50042454)%
    \put(0,0){\includegraphics[width=\unitlength]{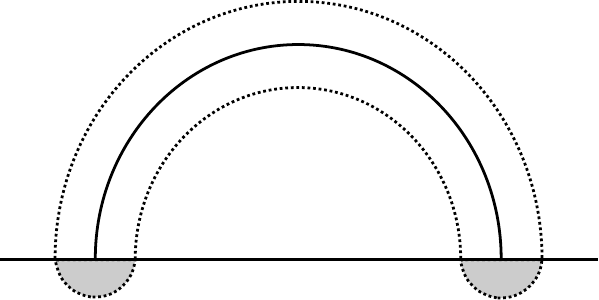}}%
    \put(0.48892202,0.44169024){\color[rgb]{0,0,0}\makebox(0,0)[b]{\smash{$\Lambda$}}}%
  \end{picture}%
\endgroup%
}
	\captionof{figure}{We have to rule out that the initial curve is in the grey areas of the $C^0$-neighborhood of the reference curve $\Lambda$.}
	\label{grey}
\end{center}\vspace{0 cm}
The next lemma solves this problem:
\begin{lem} \label{refcurve2}
	Let $\Phi^*$ and $f_0$ fulfill the assumptions of Theorem \ref{refcurve}. Then $[f_0 (\sigma)]_2 > 0$ for $\sigma \in (0, x) \cup (y, 1)$, where $x := \min\{\sigma \in [0, 1] \; | \; [\Phi^* (\sigma)]_2 = \xi_0 \}$ and $y := \max\{\sigma \in [0, 1] \; | \; [\Phi^* (\sigma)]_2 = \xi_0 \}$. In particular, the Conditions \ref{(1)} and \ref{(2)} in Lemma \ref{refcurve1} are fulfilled. 
\end{lem}

\begin{proof}[Proof of Lemma \ref{refcurve2}]
	The proof of the lemma is split into two parts: First we show that a condition on the $L_\infty$-distance of $f_0$ and $\Phi^*$ in terms of $\|\kappa_{\Lambda}\|_{C([0, 1])}$ is sufficient to guarantee $f_0([0, 1]) \subset \Psi([0, 1] \times (- \lambda d, \lambda d)) \cup \bigcup_{i=0, 1} B^{C^0}_{\bar{d}}(\Phi^*(i))$, which
	is a first step to replace Condition \ref{(1)} of Lemma \ref{refcurve1}.
	The remaining part is done afterwards: We have to assure that the initial curve is not contained in the grey parts in Figure \ref{grey}. \\
	
	W.l.o.g.\ we can assume that $\alpha \in (0, \nicefrac{\pi}{2})$: the handling of the case $\alpha \in (\nicefrac{\pi}{2}, \pi)$ will be the same as the first one, as they are symmetric. Moreover, in the case $\alpha = \nicefrac{\pi}{2}$ the bound is just given by $\lambda d$, as $\cot{\nicefrac{\pi}{2}} = 0$.
	
	Now let ${\sigma} \in [0, 1]$ be arbitrary but fixed. By rotation, we can assume that the tangent vector of $\Phi^*({\sigma})$ is horizontal. Due to the curvature bounds on the reference curve $\Phi^*$, we deduce that the curve is in the complement of two circles with radius $r :=(\|\kappa_{\Lambda}\|_{C([0, 1])})^{-1}$ touching at $\Phi^*({\sigma})$, cf.\ Figure \protect{\ref{3}}.
	\begin{center}\vspace{0cm}
		\scalebox{1}{ 
\begingroup%
  \makeatletter%
  \providecommand\color[2][]{%
    \errmessage{(Inkscape) Color is used for the text in Inkscape, but the package 'color.sty' is not loaded}%
    \renewcommand\color[2][]{}%
  }%
  \providecommand\transparent[1]{%
    \errmessage{(Inkscape) Transparency is used (non-zero) for the text in Inkscape, but the package 'transparent.sty' is not loaded}%
    \renewcommand\transparent[1]{}%
  }%
  \providecommand\rotatebox[2]{#2}%
  \ifx\svgwidth\undefined%
    \setlength{\unitlength}{234.88495942bp}%
    \ifx\svgscale\undefined%
      \relax%
    \else%
      \setlength{\unitlength}{\unitlength * \real{\svgscale}}%
    \fi%
  \else%
    \setlength{\unitlength}{\svgwidth}%
  \fi%
  \global\let\svgwidth\undefined%
  \global\let\svgscale\undefined%
  \makeatother%
  \begin{picture}(1,0.75518287)%
    \put(0,0){\includegraphics[width=\unitlength]{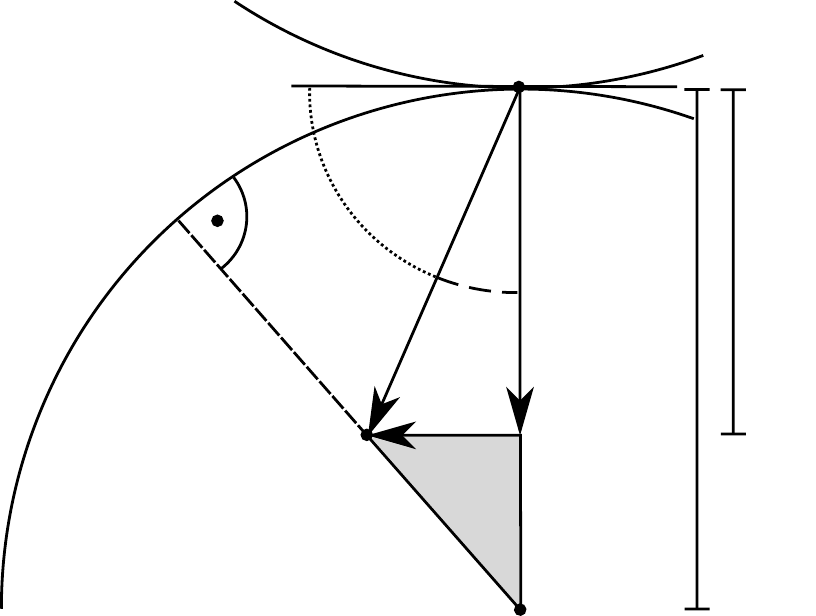}}%
    \put(0.50066496,0.57375202){\color[rgb]{0,0,0}\makebox(0,0)[b]{\smash{$\alpha$}}}%
    \put(0.62323899,0.45118595){\color[rgb]{0,0,0}\makebox(0,0)[b]{\smash{$\frac{\pi}{2} - \alpha$}}}%
    \put(0.88370856,0.11412952){\color[rgb]{0,0,0}\makebox(0,0)[lb]{\smash{$r$}}}%
    \put(0.92967389,0.4971482){\color[rgb]{0,0,0}\makebox(0,0)[lb]{\smash{$\lambda d$}}}%
    \put(0.62256312,0.17507172){\color[rgb]{0,0,0}\makebox(0,0)[b]{\smash{\small{$\lambda d \eta(\sigma) \cot \alpha \tau_{\Lambda}$}}}}%
    \put(0.27083885,0.31329926){\color[rgb]{0,0,0}\makebox(0,0)[lb]{\smash{$\bar{d}$}}}%
    \put(0.51598562,0.45663061){\color[rgb]{0,0,0}\rotatebox{66.2375064}{\makebox(0,0)[b]{\smash{\small{$\lambda d(n_{\Lambda} + \eta(\sigma) \cot \alpha \tau_{\Lambda})$}}}}}%
    \put(0.63758921,0.6733512){\color[rgb]{0,0,0}\makebox(0,0)[b]{\smash{$\Phi^*(\sigma)$}}}%
    \put(0.27663356,0.17899247){\color[rgb]{0,0,0}\makebox(0,0)[lb]{\smash{$\Psi(\sigma, \lambda d)$}}}%
    \put(0.46555558,0.07416502){\color[rgb]{0,0,0}\makebox(0,0)[b]{\smash{$r - \bar{d}$}}}%
  \end{picture}%
\endgroup%
}
		\captionof{figure}{Estimation of the $C^0$-neighborhood.}
		\label{3}
	\end{center}\vspace{0 cm}
	This implies that the distance of $\Psi({\sigma}, \lambda d)$ to the curve $\Phi^*([0, 1])$ can be bounded from below by the distance of $\Psi({\sigma}, \lambda d)$ to the circle around, which we denote by $\bar{d}$. In order to quantify $\bar{d}$, we use elementary geometry on the grey triangle in Figure \ref{3}. By Pythagoras' Theorem, we have
	\begin{align*}
	(r - \bar{d})^2 = (\lambda d \eta({\sigma}) \cot \alpha)^2 + (r - \lambda d)^2
	&& \Leftrightarrow && \bar{d} = r - \sqrt{(\lambda d \eta({\sigma}) \cot \alpha)^2 + (r - \lambda d)^2}.
	\end{align*}
	Using the definition of $d$ in \eqref{defd}, we obtain 
	\begin{align*}
	\bar{d} 
	\geq r \left(1 - \sqrt{(\lambda C_{\alpha} \cot \alpha)^2 + (1 - \lambda C_{\alpha})^2}\right)
	= r \overline{C_{\alpha}(\lambda)}.
	\end{align*}
	These considerations are a starting point to replace Condition \ref{(1)} by a $C^0$-condition: If $f_0$ and $\Phi^*$ fulfill $|f_0 (\sigma) - \Phi^* (\sigma) | < \bar{d}$ for all $\sigma \in [0, 1]$, which holds due to Condition \ref{(2)}, then it follows $f_0([0, 1]) \subset \Psi([0, 1] \times (- \lambda d, \lambda d)) \cup \bigcup_{i=0, 1} B^{C^0}_{\bar{d}}(\Phi^*(i))$. \\ \\
	In order to satisfy Condition \ref{(1)} it remains to prove that $[f_0 (\sigma)]_2$ is non negative in a neighbourhood of the boundary points, more precisely for $\sigma \in (0, x) \cup (y, 1)$, where $x := \min\{\sigma \in [0, 1] \; | \; [\Phi^* (\sigma)]_2 = \xi_0\}$ and $y := \max\{\sigma \in [0, 1] \; | \; [\Phi^* (\sigma)]_2 = \xi_0\}$. 
	Since the situations at the boundary points are the same, we concentrate on the boundary point $\sigma = 0$. By direct estimates it follows
	\begin{align*}
	[\Phi^* (\sigma)]_2
	&\geq - \int_0^{\sigma} \int_0^{{{\tilde{\sigma}}}}  \| \kappa_{\Lambda} \|_{C([0, 1])} (\mathcal{L}[\Phi^*] )^2 \dd {\bar{{{\sigma}}}} \dd {{\tilde{\sigma}}} + \sin \alpha \mathcal{L}[\Phi^*] \sigma \\
	&= - \frac{\sigma^2}{2} \| \kappa_{\Lambda} \|_{C([0, 1])} (\mathcal{L}[\Phi^*] )^2 + \sin \alpha \mathcal{L}[\Phi^*] \sigma,
	\end{align*}
	where it was used that
	\begin{align*}
	\partial_{\sigma} \Phi^* (0) = 
	\mathcal{L}[\Phi^*] \begin{pmatrix} \cos \alpha \\
	\sin \alpha \end{pmatrix}.
	\end{align*}
	We deduce
	\begin{align*}
	[\Phi^* (\sigma)]_2  > \xi_0 && \textrm{ if } 
	- \frac{\sigma^2}{2} \| \kappa_{\Lambda} \|_{C([0, 1])} (\mathcal{L}[\Phi^*] )^2 + \sin \alpha \mathcal{L}[\Phi^*] \sigma > \xi_0.
	\end{align*}
	By using the quadratic formula, the previous inequality holds true for $\sigma \in (x_-, x_+)$, where
	\begin{align*}
	x_{\mp} = \frac{\sin \alpha \mathcal{L}[\Phi^*] \mp \sqrt{(\sin \alpha \mathcal{L}[\Phi^*])^2 - 2 \| \kappa_{\Lambda} \|_{C([0, 1])} (\mathcal{L}[\Phi^*])^2 \xi_0}}{\| \kappa_{\Lambda} \|_{C([0, 1])} (\mathcal{L}[\Phi^*])^2 }. 
	\end{align*}
	For $\xi_0  < (\sin \alpha)^2 (2 \| \kappa_{\Lambda} \|_{C([0, 1])})^{-1}$, we obtain by using $a^2-b^2 > (a-b)^2$ for $a > b > 0$ to the argument of the square root that
	\begin{align*}
	x_{-} < \frac{\sqrt{2 \xi_0}}{\sqrt{\| \kappa_{\Lambda} \|_{C([0, 1])}} \mathcal{L}[\Phi^*]} := \bar{x}.
	\end{align*}
	Thus, it suffices to show $[f (\sigma)]_2 > 0$ for $\sigma \in (0, \bar{x})$: By the fundamental theorem of calculus, we have 
	\begin{align*}
	[f (\sigma)]_2 &= \int_0^{\sigma} [\partial_{{{{\sigma}}}} f_0 (\tilde{\sigma})]_2 \dd {{\tilde{\sigma}}} \geq \int_0^{\sigma} \left( [\partial_{{{\tilde{\sigma}}}} \Phi^* (\tilde{\sigma})]_2 - \xi_1 \right) \dd {{\tilde{\sigma}}} \\
	&\geq - \frac{\sigma^2}{2} \| \kappa_{\Lambda} \|_{C([0, 1])} (\mathcal{L}[\Phi^*] )^2 + \sin \alpha \mathcal{L}[\Phi^*] \sigma - \xi_1 \sigma,
	\end{align*}
	where we used $[f_0 (0)]_2 = 0$, $\partial_{\sigma} f_0 \in B^{C^0}_{\xi_1}(\partial_{\sigma}\Phi^*)$, and for the last line the same argument as for $[\Phi^* (\sigma)]_2$. The roots of the equation
	\begin{align*}
	- \frac{\sigma^2}{2} \| \kappa_{\Lambda} \|_{C([0, 1])} (\mathcal{L}[\Phi^*] )^2 + \sin \alpha \mathcal{L}[\Phi^*] \sigma - \xi_1 \sigma = 0
	\end{align*}
	are given by
	\begin{align*}
	z_{-} = 0 && \textrm{ and } && z_{+} = \frac{2(\mathcal{L}[\Phi^*]\sin \alpha - \xi_1)}{\| \kappa_{\Lambda} \|_{C([0, 1])} (\mathcal{L}[\Phi^*])^2}.
	\end{align*} 
	Note that $z_+$ is positive as $\xi_1 < \mathcal{L}[\Phi^*] \sin \alpha$. Thus, $[f (\sigma)]_2 > 0$ for $\sigma \in (z_-, z_+) = (0, z_+)$ and it remains to prove that $z+ > \bar{x}$, which is equivalent to
	\begin{align*}
          && \frac{2(\mathcal{L}[\Phi^*]\sin \alpha - \xi_1)}{\| \kappa_{\Lambda} \|_{C([0, 1])} (\mathcal{L}[\Phi^*])^2} > \frac{\sqrt{2 \xi_0}}{\sqrt{\| \kappa_{\Lambda} \|_{C([0, 1])}} \mathcal{L}[\Phi^*]}
        \end{align*}
        and which is a consequence of
        \begin{align*}
	 && \left(\mathcal{L}[\Phi^*]\sin \alpha > \sqrt{2 \xi_0 \| \kappa_{\Lambda} \|_{C([0, 1])}} \mathcal{L}[\Phi^*] \right) \; \wedge \; (\mathcal{L}[\Phi^*]\sin \alpha > 2 \xi_1).
	\end{align*}
	But the latter inequalities follow by the choice of $\xi_0$ and $\xi_1$. \\
	
	Combining this with the result from the first part shows that Condition \ref{(1)} holds true. As Condition \ref{(2)} is fulfilled by assumption, the proof is complete.
\end{proof}

\begin{proof}[Proof of Theorem \ref{refcurve}]
	The claim follows from the Lemmas \ref{refcurve1} and \ref{refcurve2}.
\end{proof}

\subsection{Some Technical Estimates \label{constructionII}}

In order to prove that the smoothed curves constructed in Section \ref{A} can be used as reference curves, we need the following estimates:
\begin{lem} \label{absch}
	Let $f_0: \bar{I} \rightarrow \mathbb{R}^2$, $I := (0, 1)$, be parametrized proportional to arc length and in $W_2^{4 (\mu - \nicefrac{1}{2})}(I; \mathbb{R}^2)$ for $\mu \in \left(\frac78, 1 \right]$.
	Moreover, let $f(t) = f(t, \cdot) \in C^5(\bar{I}; \mathbb{R}^2)$, $t \in (0, T)$, be the curves given by Lemma \ref{lemrefcurve}, cf.\ Lemma \ref{regular} for the regularity. Then it holds for $\frac{1}{6} > \delta > 0$
	\small
	\begin{align}
	\|f(\epsilon) - f_0\|_{C(\bar{I}; \mathbb{R}^2)}&\leq C \left(\epsilon^{\frac{2}{3} \mu - \frac{5}{12} - \delta}\| u_0\|_{W_2^{4\left(\mu - \nicefrac{1}{2}\right)}(I; \mathbb{R}^2)} +\epsilon^{\frac{11}{12} - \delta} \| h \|_{L_2(I; \mathbb{R}^2)} \right), \label{erste}\\
	\|f(\epsilon)\|_{C^2(\bar{I}; \mathbb{R}^2)} &\leq C \left({\epsilon^{-\frac{3}{4} + \frac{2}{3} \mu - \delta}} \|u_0\|_{W_2^{4\left(\mu - \nicefrac{1}{2}\right)}(I; \mathbb{R}^2)} +  \epsilon^{\frac{7}{12} - \delta} \| h \|_{L_2(I; \mathbb{R}^2)} + \|\xi\|_{C^2(\bar{I}; \mathbb{R}^2)}\right). \label{zweite}
	\end{align}
	\normalsize
	Additionally, we have for a sufficiently small $\delta > 0$
	\small
	\begin{align}
	\|f(\epsilon) - f_0\|_{C^1(\bar{I}; \mathbb{R}^2)}\leq C \left(\epsilon^{\frac{2}{3} \mu - \frac{7}{12} - \delta}\| u_0\|_{W_2^{4\left(\mu - \nicefrac{1}{2}\right)}(I; \mathbb{R}^2)} +\epsilon^{\frac{3}{4} - \delta} \| h \|_{L_2(I; \mathbb{R}^2)}\right). \label{dritte}
	\end{align}
	\normalsize
	Here, $\xi$, $u_0 := f_0 - \xi$ and $h:= \partial_x^6 \xi$ are quantities determined by $f_0$, see \eqref{xi}.
\end{lem}

\begin{proof}
	Since we want to take advantage of the properties of the analytic $C_0$-semigroup generated by $-A$, we switch again to the Cauchy problem \eqref{CP} which is equivalent to \eqref{sys}
	for the initial datum $u_0 = f_0 - \xi$ and the right-hand side $h(t) = h = \partial_x^6 \xi$, $\xi \in C^{\infty}(\bar{I}; \mathbb{R}^2)$ is given by \eqref{xi}. As $-A$ generates an analytic $C_0$-semigroup, cf.\ Claim 6.1.5 in \cite{butzdiss}, the mild solution formula can be exploited, see proof of Claim \ref{mildsol},
	\begin{align*}
	v(t) = e^{-tA}u_0 + \int_0^{t} e^{-sA} h \dd s.
	\end{align*}
	Additionally, we will use the following characterization of $D(A)$: It is well-known, that the norms $\|\cdot\|_{W^{6}_2(I;\mathbb{R}^2)}$ and $\|\cdot\|_{D(A)}$ are equivalent on $D(A)$, see Lemma 6.1.7 in \cite{butzdiss}. We can consider the components of $u$ separately, as they are not coupled by the equations. Applying Section 4.3.3 in \protect{\cite{triebel}} on both components of $u$, we obtain the representation
	\begin{align}
	D(A) = B^6_{2,2, \{B_j\}} (I; \mathbb{R}^2) := \left\{f \in B^6_{2,2} (I; \mathbb{R}^2): B_j f_{|\partial I} = 0 \textrm{ for } j < s - \frac{1}{2} \right\}, 
	\label{chara}
	\end{align}
	where $B_j$ are the differential operators given by
	\begin{align*}
	B_j f := b_j \partial^j_{x} f \qquad \textrm{ for } b_j =  
	\begin{cases}
	1 &\textrm{ for } j = 0, 1, 2, \\
	0 &\textrm{ for } j = 3, 4, 5.
	\end{cases}
	\end{align*}
	
	In the following, our main tool will be Proposition 2.2.9(i) in \protect{\cite{lunardi}}, which is stated here in a notation adjusted to our problem, as $-A$ (and not $A$) is the generator of the semigroup: \\
	Let $(\alpha, p)$,  $(\beta, p) \in (0, 1) \times [1, \infty] \cup \{(1, \infty)\}$, and let $n \in \mathbb{N}$. Then there are constants $C= C(n, p, \alpha, \beta)$ such that
	\begin{align}
	\left \|t^{n - \alpha + \beta} (-A)^n e^{-tA} \right\|_{L({D}_A (\alpha, p), D_A (\beta, p))} \leq C&& \textrm{ for } 0< t \leq 1.
	\label{luna}
	\end{align}
	The statement also holds for $n =0$, provided $\alpha \leq \beta$. 
	Moreover, we will use an inequality in the proof of the latter proposition:
	For $\alpha = 0$, we set $D_A (\alpha, p) = X$, cf.\ Remark before Proposition 2.2.9 in \protect{\cite{lunardi}}. Then, we have
	\begin{align}
	\left \|t^{n+ \beta} (-A)^n e^{-tA} \right\|_{L(X, D_A (\beta, p))} \leq C &&  \textrm{ for }  0< t \leq 1.
	\label{lunaproof1}
	\end{align}
	We start proving the estimate \eqref{erste}. Clearly, it holds
	\begin{align*}
	f(\epsilon) - f_0 = (f(\epsilon) - \xi) - (f_0 - \xi) = v(\epsilon) - u_0.
	\end{align*}
	Applying Lemma \ref{unab0b}, we have for $\delta > 0$
	\begin{align*}
	W_2^{\nicefrac12 + 6 \delta}\left(I; \mathbb{R}^2\right) \hookrightarrow C^{\gamma}\left(\bar{I}; \mathbb{R}^2\right)  \hookrightarrow C\left(\bar{I}; \mathbb{R}^2\right),
	\end{align*}
	for $0 < \gamma < 6 \delta$. Thus, it follows
	\begin{align*}
	\|f(\epsilon) - f_0 \|_{C(\bar{I}; \mathbb{R}^2)} \leq  C \|v(\epsilon) - u_0\|_{W^{\nicefrac{1}{2} + 6 \delta}_2(I; \mathbb{R}^2)}.
	\end{align*}
	Choosing $\beta = \nicefrac{1}{12} + \delta$, for an arbitrary $\nicefrac{1}{6} > \delta > 0$ ,
	we obtain
	\begin{align}
	D_A (\beta, 2) = \left(X, D(A) \right)_{\beta, 2} = \left\{u \in W_2^{\nicefrac12 + 6 \delta}\left(I; \mathbb{R}^2\right) : u_{|\partial I} = 0 \right\}
	\label{genauwie}
	\end{align}
	with equivalent norms. Here we used Proposition 2.2.2 in \protect{\cite{lunardi}} for the first equality and the combination of \eqref{chara} and Theorem in Section 4.3.3 of \protect{\cite{triebel}} for the second identity. 
	Consequently, we have
	\begin{align*}
	\|f(\epsilon) - f_0 \|_{C(\bar{I}; \mathbb{R}^2)} \leq  C \|v(\epsilon) - u_0\|_{ D_A (\beta, 2)}.
	\end{align*}
	Using the mild solution formula and the triangle inequality, we obtain
	\begin{align*}
	\|f(\epsilon) - f_0\|_{C(\bar{I}; \mathbb{R}^2)} 
	\leq C \left( \left\|\int_0^{\epsilon} -A e^{-sA} u_0 \dd s \right\|_{ D_A (\beta, 2)} + \left\|\int_0^{\epsilon} e^{-sA} h \dd s \right\|_{ D_A (\beta, 2)} \right) = C(I + II),
	\end{align*}
	where we used $e^{-\epsilon A}u_0 - u_0 = \int_0^{\epsilon} -A e^{-sA} u_0 \dd s$, cf.\ Proposition 2.1.4 (ii) in \protect{\cite{lunardi}}. For the first summand, we obtain by \eqref{luna} in case $n = 1$, $\beta = \nicefrac{1}{12} + \delta$, $\alpha = \tilde{\mu}- \nicefrac{1}{2}$
	\begin{align*}
	I 
	\leq \int_0^{\epsilon} \left\|A e^{-sA} \right\|_{L({D}_A (\alpha, 2), D_A (\beta, 2))} \| u_0\|_{ D_A (\alpha, 2)} \dd s 
	\leq \int_0^{\epsilon} \frac{C}{s^{1 - (\tilde{\mu}- \nicefrac{1}{2}) + (\nicefrac{1}{12} + \delta)} } \| u_0\|_{ D_A (\alpha, 2)} \dd s.
	\end{align*}
	Here it follows analogously to the explanation of \eqref{genauwie} that
	\begin{align*}
	&\left\{ u \in W_2^{4\left(\mu - \nicefrac{1}{2}\right)}(I; \mathbb{R}^2):     \; u_{|\partial I} = 0, \partial_x  u_{|\partial I} = 0  \right\} \\
	& \qquad = \left\{u \in W_2^{6(\tilde{\mu} - \nicefrac{1}{2})}(I; \mathbb{R}^2) : \; u_{|\partial I} = 0, \partial_x  u_{|\partial I} = 0 \right\} = D_A (\alpha, 2)
	\end{align*}
	with equivalent norms for $\mu \in \left(\nicefrac78, 1 \right]$ and $\tilde{\mu} \in (\nicefrac{3}{4}, \nicefrac{5}{6}]$ with $\tilde{\mu} = \nicefrac{2}{3} \mu + \nicefrac{1}{6}$, respectively. Since we have
	\begin{align*}
	\left(\tilde{\mu}- \frac{1}{2}\right) - \left(\frac{1}{12} + \delta \right) = \tilde{\mu}- \frac{7}{12} - \delta > \frac{3}{4}- \frac{7}{12} - \frac{1}{6} = 0 && \textrm{ for } 0 < \delta < \nicefrac{1}{6},
	\end{align*}
	thus $1 - (\tilde{\mu}- \nicefrac{1}{2}) + (\nicefrac{1}{12} + \delta) < 1 $,
	we can integrate with respect to $s$ and obtain
	\begin{align*}
	I \leq C \epsilon^{\tilde{\mu} - \frac{7}{12} - \delta} \| u_0\|_{W_2^{4\left(\mu - \nicefrac{1}{2}\right)}(I; \mathbb{R}^2)}.
	\end{align*}
	Using \eqref{lunaproof1} in the case $n = 0$, $\beta = \nicefrac{1}{12} + \delta$, $\alpha = 0$, we deduce for the second summand
	\begin{align*}
	II 
	\leq \int_0^{\epsilon} \left\| e^{-sA}\right\|_{L(X, D_A (\beta, 2))} \| h \|_{X} \dd s \leq \int_0^{\epsilon} \frac{C}{s^{\beta}} \| h \|_{X} \dd s 
	\leq C \epsilon^{\frac{11}{12} - \delta} \| h \|_{X}.
	\end{align*}
	In summary, we obtain
	by $\tilde{\mu}(\mu) = \nicefrac{2}{3} \mu + \nicefrac{1}{6}$
	\begin{align*}
	\|f(\epsilon) - f_0\|_{C(\bar{I}; \mathbb{R}^2)}\leq C \left(\epsilon^{\frac{2}{3} \mu - \frac{5}{12} - \delta}\| u_0\|_{W_2^{4\left(\mu - \nicefrac{1}{2}\right)}(I; \mathbb{R}^2)} +\epsilon^{\frac{11}{12} - \delta} \| h \|_{X}\right),
	\end{align*}
	where $\nicefrac{2}{3} \mu - \nicefrac{5}{12} - \delta > \nicefrac{2}{3} \cdot \nicefrac{7}{8} - \nicefrac{5}{12} - \delta =  \nicefrac{1}{6} - \delta > 0$ for $\mu \in \left(\nicefrac78, 1 \right]$. \\
	
For the proof of the inequalities \eqref{zweite} and \eqref{dritte} the same techniques are employed, see the proof of Lemma 6.3.1 in \cite{butzdiss} for details.
\end{proof}

\subsection{$f_{\epsilon}$ is a Reference Curve \label{feps}}
The technical estimates from the previous section enable us to apply the results of Subsection \ref{param} to the curves $f_{\epsilon} = f(\epsilon, \cdot )$,  $\epsilon>0$, derived in Section \ref{A}. \\

The main result reads as follows:
\begin{satz} [$f_{\epsilon}$ is a reference curve for $f_0$ provided $\epsilon$ is small enough] \label{mtfe}
	Let $f_0: \bar{I} \to \mathbb{R}^2$, $I := (0, 1)$, be parametrized proportional to arc length, let it be in $W_2^{4\left(\mu - \nicefrac{1}{2}\right)}(I; \mathbb{R}^2)$, $\mu \in \left(\frac{7}{8}, 1 \right]$, and let it fulfill the boundary conditions given in \eqref{initialdatum}.
Moreover, let $f(\epsilon, \cdot) := f_{\epsilon}: \bar{I} \to \mathbb{R}^2$, $\epsilon > 0$, be the smoothed curves generated by evolving $f_0$ by a parabolic equation, see Lemma \ref{lemrefcurve} and Lemma \ref{regular} in Subsection \ref{A}.\\
	Then there exists an $\tilde{\epsilon} > 0$, such that for $f_{\epsilon}$, $0 < \epsilon < \tilde{\epsilon}$, the conditions in Theorem \ref{refcurve} are fulfilled for the parameterizations $f_{\epsilon} \circ \beta_{\epsilon}$ and $f_0 \circ \beta_{\epsilon}$. Here, $\beta_{\epsilon}: \bar{I} \rightarrow \bar{I}$ is the orientation preserving reparametrization such that $f_{\epsilon} \circ \beta_{\epsilon}$ is parametrized proportional to arc length. In particular, $f_\epsilon$ is a reference curve for the initial curve $f_0$.
\end{satz}

Note that it is not trivial that $f_\epsilon$ is a reference curve for the initial curve $f_0$ for some $\epsilon>0$: Although, by making $\epsilon$ smaller, we can diminish the $C^0$-distance between $f_0$ and $f_\epsilon$, see \eqref{erste}, but the curvature $\kappa$ of $f_\epsilon$ may explode, see \eqref{zweite}. Therefore, we have to be careful, since the bound on the $C^0$-distance is proportional to the reciprocal of the $C^0$-norm of the curvature of the reference curve, cf.\ \eqref{defd}. In order to compare those two effects, we give a formulation of  Lemma \ref{refcurve} in the "initial curve perspective", i.e.\ in the proportional-to-arc-length-parametrization on $[0, 1]$ of the initial curve. 
\begin{lem} \label{refcurve3}
	Let $f_0$ and $f_{\epsilon}$ be as in Theorem \ref{mtfe}. Then there exists an $\bar{\epsilon} > 0$ such that $f_\epsilon$, $0 < \epsilon < \bar{\epsilon}$, is a regular curve. Let $\lambda \in (0,1)$ be given such that the conditions \eqref{lam} and \eqref{lambdasmall} are fulfilled. If $f_{\epsilon} \in B^{C^0}_{\xi_0}(f_0)$ and $\partial_{\sigma} f_{\epsilon} \in B^{C^0}_{\xi_1}(\partial_{{\sigma}} f_0)$ for $0 < \epsilon < \bar{\epsilon}$ and
	\begin{align*}
	\xi_0 &= \min \left\{\overline{C_{\alpha}(\lambda)}, \frac{(\sin \alpha)^2}{2} \right\} \frac{1}{\|\kappa [f_\epsilon]\|_{C([0, 1])}}, \\
	\xi_1 &= \min \left\{\frac{\sqrt{(\cot \alpha)^2 +1} - |\cot \alpha|}{4}, 
	\left\{\begin{array}{c}
	\left.\begin{aligned}
	&\tfrac{1}{2(4 \cdot 144)^2 |\cot \alpha|^2} &\textrm{ for } \alpha \neq \tfrac{\pi}{2} \\
	& 1 &\textrm{ for } \alpha = \tfrac{\pi}{2} 
	\end{aligned}\right\}
	\end{array}\right. \hspace{-0,25cm} , \frac{\sin \alpha}{2} \right\} K(\bar{\epsilon}, f_0),
	\end{align*}
	where 
	\begin{align*}
	K(\bar{\epsilon}, f_0) : = \mathcal{L}[f_0] - C(\bar{\epsilon}) > 0
	\end{align*}
	for $C(\epsilon) \rightarrow 0$ monotonically for $\epsilon \rightarrow 0$, and $\overline{C_{\alpha}(\lambda)}$ is defined in Lemma \ref{refcurve}, then $f_\epsilon$ is a reference curve for the initial curve $f_0$.
\end{lem}

\begin{proof}
	The strategy is to go through the proof of Theorem \ref{refcurve}, to reparametrize the inequalities and to replace the conditions. We recall that reparametrization of the condition for $f_{\epsilon}$ and $f_0$ does not affect the radius $\xi_0$. Furthermore, we set
	\begin{align*}
	\widetilde{\xi_1} &:= \min \left\{\frac{\sqrt{(\cot \alpha)^2 +1} - |\cot \alpha|}{4}, 
	\left\{\begin{array}{c}
	\left.\begin{aligned}
	&\tfrac{1}{2(4 \cdot 144)^2 |\cot \alpha|^2} &\textrm{ for } \alpha \neq \tfrac{\pi}{2} \\
	& 1 &\textrm{ for } \alpha = \tfrac{\pi}{2} 
	\end{aligned}\right\}
	\end{array}\right. \hspace{-0,25cm} , \frac{\sin \alpha}{2} \right\}
	\end{align*}
	Then, the condition for the derivatives of $f_{\epsilon}$ and $f_0$ is given by
	\begin{align*}
	|\partial_{\sigma} (f_0 \circ \beta_{\epsilon}) (\sigma) - \partial_{\sigma} (f_{\epsilon}\circ \beta_{\epsilon}) (\sigma) | &< \widetilde{\xi_1} |\partial_{\sigma} (f_{\epsilon}\circ \beta_{\epsilon}) (\sigma)|,
	\end{align*}
	supposed $\beta_{\epsilon}: \bar{I} \rightarrow \bar{I}$ is a regular orientation preserving reparametrization such that $f_{\epsilon} \circ \beta_{\epsilon}$ is parametrized proportional to arc length. By chain rule, this is equivalent to 
	\begin{align}
	|\partial_{\sigma} f_0  (\beta_{\epsilon}(\sigma)) - \partial_{\sigma} f_{\epsilon}(\beta_{\epsilon}(\sigma)) | &< \widetilde{\xi_1} |\partial_{\sigma} f_{\epsilon} (\beta_{\epsilon}(\sigma))|,
	\label{neu}
	\end{align}
	as $\beta_{\epsilon}'(\sigma) > 0$ for $\sigma \in [0, 1]$. Thus, it remains to prove that $f_{\epsilon}$ is a regular curve and that there exists a uniform lower bound for $|\partial_{\sigma} f_{\epsilon}|$ for every $0 < \epsilon < \bar{\epsilon}$.
	
	\begin{claim}
		There exists an $\bar{\epsilon} > 0$ such that for every $0 < \epsilon < \bar{\epsilon}$ the following holds true:
		The function $f_{\epsilon}(\cdot)$ is regular with the bound
		\begin{align}
		K(\bar{\epsilon}, f_0) \leq |\partial_\sigma f_{\epsilon}(\sigma)|
		\label{derbound}
		\end{align}
		and there exists a uniform lower bound on $\mathcal{L}[f_{\epsilon}]$
		\begin{align}
		K(\bar{\epsilon}, f_0) \leq \mathcal{L}[f_{\epsilon}],
		\label{lengthbound}
		\end{align}
		where $K(\bar{\epsilon}, f_0) : = \mathcal{L}[f_0] - C(\bar{\epsilon}) > 0$ with $C(\epsilon) \rightarrow 0$ as $\epsilon \rightarrow 0$.
	\end{claim}
	
	\begin{proof}[Proof of the Claim:]
		We want to use estimate \eqref{dritte} in Lemma \ref{absch}: For $\mu \in (\nicefrac{7}{8}, 1]$ it holds $\nicefrac{2}{3} \mu - \nicefrac{7}{12} > 0$
		thus, the first summand of the right-hand side in \eqref{dritte} has a positive $\epsilon$ power if $\delta$ is chosen sufficiently small. This implies that 
		\begin{align*}
		\mathcal{L}[f_0] - C(\epsilon) = \min_{\sigma \in [0, 1]} |\partial_\sigma f_0 (\sigma)| - C(\epsilon) \leq \min_{\sigma \in [0, 1]} |\partial_\sigma f_{\epsilon}(\sigma)|
		\end{align*}
		with a $C(\epsilon) \rightarrow 0$ monotonically as $\epsilon \rightarrow 0$. Choosing $\bar{\epsilon}$ sufficiently small, we obtain
		\begin{align*}
		0 < \mathcal{L}[f_0] - C(\bar{\epsilon}) \leq \mathcal{L}[f_0] - C(\epsilon) \leq \min_{\sigma \in [0, 1]} |\partial_\sigma f_{\epsilon}(\sigma)| \leq |\partial_\sigma f_{\epsilon}(\sigma)|
		\end{align*}
		for $\sigma \in [0, 1]$ and all $0 < \epsilon < \bar{\epsilon}$, which shows the estimate \eqref{derbound} and that the parametrization is regular. Integrating the last inequality with respect to the parameter $\sigma$ over $[0, 1]$, we deduce estimate \eqref{lengthbound}
		\begin{align*}
		0 < \mathcal{L}[f_0] - C(\bar{\epsilon}) = \int_{[0, 1]} \mathcal{L}[f_0] - C(\bar{\epsilon}) \dd \sigma \leq \int_{[0, 1]} |\partial_\sigma f_{\epsilon}(\sigma)| \dd \sigma = \mathcal{L}[f_{\epsilon}] 
		\end{align*}
		for all $0 < \epsilon < \bar{\epsilon}$. 
	\end{proof}
	Thus, the condition on $\xi_1$ in Lemma \ref{refcurve3}
	is stronger than \eqref{neu} and the lemma is proven.
\end{proof}

We proceed with the proof of Theorem \ref{mtfe}:
\begin{proof}[Proof of Theorem \ref{mtfe}]
	By estimate \eqref{dritte}, we see that $\|f(\epsilon) - f_0\|_{C^1(\bar{I}; \mathbb{R}^2)} \rightarrow 0$ as $\epsilon \rightarrow 0$. Thus, by choosing $\bar{\epsilon}$ small enough, the conditions on the derivatives of $f_{\epsilon}$ and $f_0$ in Lemma \ref{refcurve3} are fulfilled for $0 < \epsilon < \bar{\epsilon}$. It just remains to show the following claim:
	\begin{claim} \label{cl1}
		Let $\lambda \in (0,1)$ be given such that the conditions \eqref{lam} and \eqref{lambdasmall} are fulfilled. Then there exists an $\tilde{\epsilon}$ with $0 < \tilde{\epsilon} \leq \bar{\epsilon}$, such that $f_{\epsilon} \in B^{C^0}_{\xi_0}(f_0)$ for all $0 < \epsilon < \tilde{\epsilon}$, where $\xi_0$ is as in Lemma \ref{refcurve3}.
	\end{claim}
	
	\begin{proof}[Proof of the claim:]
		We have to show that there exists an $\tilde{\epsilon}$, such that
		\begin{align*}
		\|f_0 - f_{\epsilon}\|_{C ([0,1]; \mathbb{R}^2)} < \min \left\{\overline{C_{\alpha}(\lambda)}, \frac{(\sin \alpha)^2}{2} \right\} \frac{1}{\|\kappa [f_{\epsilon}]\|_{C ([0,1])}}
		\end{align*}
		for $0 < \epsilon < \tilde{\epsilon}$.
		To this end, we use the estimate \eqref{erste}, i.e.\
		\begin{align*}
		\|f_0 - f_{\epsilon}\|_{{C ([0,1]; \mathbb{R}^2)}} \leq C_1 \left(\epsilon^{\frac{2}{3} \mu - \frac{5}{12} - \delta_1}\| u_0\|_{W_2^{4\left(\mu - \nicefrac{1}{2}\right)}(I; \mathbb{R}^2)} +\epsilon^{\frac{11}{12} - \delta_1} \| h \|_{X} \right)
		\end{align*}
		for $0 < \delta < \frac{1}{6}$ and
		a rearranged version of \eqref{zweite} given by
		\begin{align*}
		\frac{1}{\|\kappa [f_{\epsilon}]\|_{{C ([0,1]; \mathbb{R}^2)}}} \geq C_2^{-1} \left({\epsilon^{-\frac{3}{4} + \frac{2}{3} \mu - \delta_2}}\|u_0\|_{W_2^{4\left(\mu - \nicefrac{1}{2}\right)}(I; \mathbb{R}^2)} +  \epsilon^{\frac{7}{12} - \delta_2} \| h \|_{X} + \|\xi\|_{C^2(I; \mathbb{R}^2)} \right)^{-1},
		\end{align*}
		where the $\delta_2>0$ is a sufficiently small number and the norms are finite and do not depend on $\epsilon$. Thus, it is enough to show that there exists an $\tilde{\epsilon} \in (0, \bar{\epsilon}]$ such that the inequality
		\begin{align}
		\left( \epsilon^{\frac{2}{3} \mu - \frac{5}{12} - \delta_1} \right. & \left. \| u_0\|_{W_2^{4\left(\mu - \nicefrac{1}{2}\right)}(I; \mathbb{R}^2)} +\epsilon^{\frac{11}{12} - \delta_1} \| h \|_{X} \right) \nonumber \\ 
		&\times \left( {\epsilon^{-\frac{3}{4} + \frac{2}{3} \mu - \delta_2}} \|u_0\|_{W_2^{4\left(\mu - \nicefrac{1}{2}\right)}(I; \mathbb{R}^2)} +  \epsilon^{\frac{7}{12} - \delta_2} \| h \|_{X} + \|\xi\|_{C^2(I; \mathbb{R}^2)} \right) < C
		\label{absch1}
		\end{align}
		is fulfilled for each $0 < \epsilon < \tilde{\epsilon}$, where $C:= C_1^{-1} \min \left\{\overline{C_{\alpha}(\lambda)}, \nicefrac{(\sin \alpha)^2}{2} \right\} C_2^{-1}$. Direct calculations show that the powers in $\epsilon$ of the first factor are both positive, and that $\nicefrac{2}{3} \mu - \nicefrac{5}{12} - \delta_1$ is the smaller than one. Concerning the second factor, the first summand is the only critical one, as its power in $\epsilon$ is negative for $\mu \in (\nicefrac{7}{8},1]$. We want to make sure that the product of these worst factors has in total a positive power. Thus, we calculate $\nicefrac{2}{3} \mu - \nicefrac{5}{12} - (\nicefrac{3}{4} - \nicefrac{2}{3} \mu ) > 0 $
		and deduce the existence of $\delta_1, \delta_2 > 0$, such that
		\begin{align*}
		\frac{2}{3} \mu - \frac{5}{12} - \delta_1 - \left(\frac{3}{4} - \frac{2}{3} \mu + \delta_2 \right) > 0.
		\end{align*}
		This implies that the smallest power of $\epsilon$ of the summands on the left-hand side of \eqref{absch1} is positive for a $\mu \in (\nicefrac{7}{8},1]$ provided $\delta_1, \delta_2 > 0$ are sufficiently small. Consequently, there exists $\tilde{\epsilon} >0$ such that the inequality \eqref{absch1} is fulfilled for each $0 < \epsilon < \tilde{\epsilon}$. 
	\end{proof}
	By Claim \ref{cl1}, it follows that both ball conditions in Lemma \ref{refcurve3} are fulfilled for $0 < \epsilon < \bar{\epsilon}$, if $\bar{\epsilon} > 0$ is chosen small enough. Since the conditions in Lemma \ref{refcurve3} are stronger than the ones from Theorem \ref{refcurve}, the latter are also satisfied. It is a direct consequence that $f_\epsilon$ is a reference curve for the initial curve $f_0$.
\end{proof}

This enables us to prove the theorem on local well-posedness for a fixed initial curve.
\begin{proof}[Proof of Theorem \ref{localbetter}]
	Let $f_0: \bar{I} \to \mathbb{R}^2$, $I := (0, 1)$ fulfill the assumption of Theorem \ref{localbetter}. Then, we obtain by Theorem \ref{mtfe} a reference curve $\Phi^* = f_{\epsilon} \circ \beta: [0, 1] \rightarrow \mathbb{R}^2$, which is parametrized proportional to arc length, and a corresponding initial height function $\rho_0$, which fulfill the conditions given in Definition \ref{refcurvedef}. In particular, there exists a regular $C^1$-reparametrization $\varphi: [0, 1] \rightarrow [0,1]$ and a function $\rho_0: [0, 1] \rightarrow (-d, d)$ in $W_2^{4\left(\mu - \nicefrac{1}{2}\right)}(I)$, $\mu \in \left(\nicefrac{7}{8}, 1 \right]$, such that 
	\begin{align*}
	f_0(\varphi(\sigma)) = \Phi^*(\sigma) + \rho_0(\sigma) (n_{\Lambda}(\sigma) + {\cot{\alpha}} \eta(\sigma) \tau_{\Lambda} (\sigma)), 
	\end{align*}
	where $\Lambda = \Phi^*([0, 1])$. By Theorem \ref{local} 
	, we obtain a strong solution $(t, \sigma) \mapsto \Psi(\sigma, \rho(t, \sigma))$ to \eqref{0}-\eqref{3blow} with initial curve $\Psi(\cdot, \rho_0(\cdot))$.
	By construction, compare the coordinates \eqref{curvilin} to the formula \eqref{para},
	we observe that $f_0(\varphi(\cdot)) = \Psi(\cdot, \rho_0(\cdot))$. This shows the existence of a solution. 
\end{proof}

\section{The Proof of the Blow-up Criterion Theorem \ref{c} \label{main}}

The proof is done in three steps. \\
	
	\hspace{-0.35cm}\textit{\uline{Step 1:} $W^2_2$-bound for the reparametrized and translated solution \\}
	Conversely, we assume that there exists a sequence in time $(t_l)_{l \in \mathbb{N}}$ with $t_l \to T_{max}$ as $l \to \infty$, such that $\kappa[f(t_l)]: [0, \mathcal{L}[f(t_l)]] \rightarrow \mathbb{R}$ satisfies an $L_2$-bound with respect to the arc length, which is uniform in $l \in \mathbb{N}$, i.e.\
	\begin{align}
	\left\| \kappa[f(t_l)]\right\|_{L_2(0, \mathcal{L}[f(t_l)])} \leq C \qquad \textrm{ for all } l \in \mathbb{N}.
	\label{kb}
	\end{align} 
	Here, $f: [0, T_{max}) \times \bar{I} \rightarrow \mathbb{R}^2$, $I = (0, 1)$, $T_{max} < \infty$, is a maximal solution of \eqref{0}--\eqref{3blow}, which is given by assumption.
	Let
	\begin{align*}
	\bar{I} \ni s \mapsto \sigma_l(s) \in \bar{I}
	\end{align*}
	be the orientation preserving reparametrization such that $f(t_l, \sigma_l(s)): \bar{I} \rightarrow \mathbb{R}^2$ is parame\-trized proportional to arc length.
	Then, we denote by 
	\begin{align*}
	\tau (t_l, \sigma_l(s)) := \frac{\partial_s f(t_l, \sigma_l(s))}{\mathcal{L}[f(t_l)]} && \textrm{ and }&& \vec{\kappa}[f(t_l)](\sigma_l(s)) := \frac{\partial^2_s f(t_l, \sigma_l(s))}{(\mathcal{L}[f(t_l)])^2}
	\end{align*}
	the tangent vector and the curvature vector of $f(t_l, \bar{I})$ at $f(t_l, \sigma(s))$, respectively. Now, we define 
	\begin{align*}
	\tilde{f}(t_l, \cdot): \bar{I}  \rightarrow \mathbb{R}^2,\; s \mapsto \tilde{f}(t_l, s) := \mathcal{L}[f(t_l)] \left(\int_{0}^{s} \int_{0}^{\tilde{s}} \vec{\kappa}[f(t_l)](\sigma_l(y)) \mathcal{L}[f(t_l)] \dd y + \tau(t_l, 0) \dd \tilde{s}\right).
	\end{align*}
	It follows by direct calculations that $s \mapsto \tilde{f}(t_l, s)$ is -- up to translation by $f(t_l, 0)$ -- for each $l \in \mathbb{N}$ the orientation preserving reparametrization of $f(t_l, \cdot)$ on $\bar{I}$ which is proportional to arc length. Moreover, we deduce the bounds
	\begin{align*}
	\left\|[s \mapsto \partial_{s} \tilde{f}(t_l, s) ] \right\|_{L_2(I; \mathbb{R}^2)} &= \mathcal{L}[f(t_l)]^\frac{1}{2}, \nonumber \\
	\left\|[s \mapsto \partial^2_{s} \tilde{f}(t_l, s) ] \right\|_{L_2(I; \mathbb{R}^2)} 
	&=  \mathcal{L}[f(t_l)]^\frac{3}{2} \left\| \kappa[f(t_l)]\right\|_{L_2(0, \mathcal{L}[f(t_l)])}. 
	\end{align*}
	By enlarging the domain of integration, a change of variables with $z = \mathcal{L}[f(t_l)] y$, and Hölder's inequality, we obtain
	\begin{align*}
	\Bigg\|\Bigg[ s \mapsto \int_{0}^{s} & \int_{0}^{\tilde{s}} \vec{\kappa}[f(t_l)](\sigma_l(y)) \mathcal{L}[f(t_l)] \dd y \dd \tilde{s} \Bigg] \Bigg\|_{L_2(I; \mathbb{R}^2)} 
	\leq \mathcal{L}[f(t_l)]^\frac{1}{2} \left\| \kappa[f(t_l)]\right\|_{L_2(0, \mathcal{L}[f(t_l)])}.
	\end{align*}
	Additionally, we deduce
	\begin{align*}
	\left\|\left[ {s} \mapsto \int_{0}^{{s}} \tau(t_l, 0) \dd \tilde{s} \right] \right\|_{L_2(I; \mathbb{R}^2)} \leq 1. 
	\end{align*}
	Combining these bounds with the bound 
	in Remark \ref{lengthbounded}, we obtain
	\begin{align}
	\| \tilde{f}(t_l, \cdot) \|_{W^2_2(I; \mathbb{R}^2)} \leq C_* && \textrm{ for each } l \in \mathbb{N},
	\label{w2bound}
	\end{align} 
	where we used the uniform in time bound \eqref{kb}. \\
	
	\hspace{-0.35cm}\textit{\uline{Step 2:} Restarting the flow for translated initial data\\}
	We set $\tilde{f}_l :=  \tilde{f}(t_l, \cdot)$. The bound \eqref{w2bound} implies,
	\begin{align}
	\big\|\tilde{f}_l \big\|_{W^2_2\left(I; \mathbb{R}^2 \right)} \leq C_* &&\textrm{ for all } l \in \mathbb\mathbb{N}.
	\label{boundf}
	\end{align}
	Thus, we observe that by
	\begin{align*}
	M := \left\{\tilde{f}_l : \; l \in \mathbb{N} \right\},
	\end{align*} 
	we have a bounded set in $W^2_2(I; \mathbb{R}^2)$. By combining the statements in Theorem 2 (b) in Section 1.16.4 and Theorem 1 in 4.3.1, both in \protect{\cite{triebel}}, with Theorem 10.9 (2) in \protect{\cite{alt}}, we have the compact embedding
	\begin{align}
	W^2_2 \left(I; \mathbb{R}^2 \right) \hookrightarrow W^\gamma_2\left(I; \mathbb{R}^2 \right) && \textrm{ for } \gamma < 2. \label{compe}
	\end{align}
	Consequently, the set $M$ is precompact in $W^\gamma_2\left(I; \mathbb{R}^2 \right)$. Note, that for a fixed $\gamma \in (\nicefrac{3}{2}, 2)$ we find a $\mu \in (\nicefrac{7}{8}, 1)$ such that $\gamma = 4(\mu - \nicefrac{1}{2})$. \\ 
	
	In the following, we want to find a covering for the closure of $M$ with respect to $\|\cdot\|_{W^\gamma_2(I; \mathbb{R}^2)}$. By Theorem \ref{mtfe}, there exists for each $\tilde{f}_l $ a reference curve $\Phi^*_l: \bar{I} = [0, 1] \rightarrow \mathbb{R}^2$ with the following properties:
	\begin{itemize}
		\item $\Phi^*_l$ is a regular curve and in $C^5 (\bar{I}; \mathbb{R}^2)$, see Lemma \ref{regular}.
		\item $\Phi^*_l$ fulfills \eqref{bound1},
		where $\Lambda_l = \Phi^*_l(\bar{I})$.
		\item Let $\beta_l: \bar{I} \rightarrow \bar{I}$ be the orientation preserving reparametrization such that $\Phi^*_l \circ \beta_l$ is parametrized proportional to arc length. Then, it holds that
		\begin{align*}
		\|\tilde{f}_l\circ \beta_l  - \Phi^*_l \circ \beta_l\|_{C^0 (\bar{I}; \mathbb{R}^2)} < \xi_{0, l},\\ 
		\|\partial_{{\sigma}} (\tilde{f}_l\circ \beta_l)  - \partial_{{\sigma}} (\Phi^*_l \circ \beta_l)\|_{C^0 (\bar{I}; \mathbb{R}^2)} < \xi_{1, l},
		\end{align*}
		where $\xi_{0, l}$, $\xi_{1, l}$ are defined as $\xi_{0}$, $\xi_{1}$ in Theorem \ref{refcurve} with respect to $\Phi^*_l \circ \beta_l$ instead of $\Phi^*_l$.
	\end{itemize}
	Now, we set 
	\begin{align*}
	\delta_{0, l} &:= \xi_{0, l}- \|\tilde{f}_l\circ \beta_l  - \Phi^*_l \circ \beta_l\|_{C^0 (\bar{I}; \mathbb{R}^2)} > 0,\\
	\delta_{1, l} &:= \xi_{1, l}- \|\partial_{{\sigma}} (\tilde{f}_l\circ \beta_l)  - \partial_{{\sigma}} (\Phi^*_l \circ \beta_l)\|_{C^0 (\bar{I}; \mathbb{R}^2)} > 0,
	\end{align*} 
	and consider the balls $B^{W_2^\gamma(I; \mathbb{R}^2)}(\tilde{f}_l \circ \beta_l , \nicefrac{\min_{i=1, 2} \delta_{i, l}}{2 C})$, which are balls in $W^\gamma_2(I; \mathbb{R}^2)$ around $\tilde{f}_l \circ \beta_l$ with radius $\nicefrac{\min_{i=1, 2} \delta_{i, l}}{2 C}$. Here, the constant denoted by $C$ is the operator norm of the embedding $	i: W_2^\gamma (I; \mathbb{R}^2) \hookrightarrow C^1 (\bar{I}; \mathbb{R}^2 ).$
	We can cover $\overline{M}^{W^\gamma_2(I; \mathbb{R}^2)}$ by the union of all these balls.
	By compactness, there exists a finite set $S \subset \mathbb{N}$ such that it holds
	\begin{align*}
	\overline{M}^{W^\gamma_2(I; \mathbb{R}^2)} \subset \bigcup_{l \in S} B^{W^\gamma_2(I; \mathbb{R}^2)}\left( \tilde{f}_l \circ \beta_l , \frac{\min_{i=1, 2} \delta_{i, l}}{2 C} \right).
	\end{align*}
	Therefore, for each $k \in \mathbb{N}$ there exists an $l \in S$ with $\tilde{f}_k \circ \beta_k \in B^{W^\gamma_2(I; \mathbb{R}^2)}(\tilde{f}_l \circ \beta_l , \nicefrac{\min_{i=1, 2} \delta_{i, l}}{2 C})$ again. Consequently, we have
	\begin{align*}
	\left\|\tilde{f}_k \circ \beta_k - \tilde{f}_l \circ \beta_l \right\|_{C (\bar{I}; \mathbb{R}^2)} &< C \left\|\tilde{f}_k \circ \beta_k - \tilde{f}_l \circ \beta_l \right\|_{W^\gamma_2(I; \mathbb{R}^2)} < \frac{\delta_{0, l}}{2} \\
	\left\|\partial_{{\sigma}} (\tilde{f}_k\circ \beta_k)  - \partial_{{\sigma}} (\tilde{f}_l \circ \beta_l) \right\|_{C (\bar{I}; \mathbb{R}^2)} & < \frac{\delta_{1, l}}{2}.
	\end{align*}
	These estimates imply by adding a zero
	\small
	\begin{align*}
	\big\| \tilde{f}_k \circ \beta_k - \Phi^*_l \circ \beta_l \big\|_{C (\bar{I}; \mathbb{R}^2)} 
	&\leq \frac{1}{2} \left(\xi_{0, l} - \left\|\tilde{f}_l \circ \beta_l - \Phi^*_l \circ \beta_l \right\|_{C (\bar{I}; \mathbb{R}^2)}\right) + \left\|\tilde{f}_l\circ \beta_l  - \Phi^*_l \circ \beta_l \right\|_{C (\bar{I}; \mathbb{R}^2)}\\
	&\leq \frac{1}{2} \left(\xi_{0, l} + \left\|\tilde{f}_l\circ \beta_l  - \Phi^*_l \circ \beta_l \right\|_{C (\bar{I}; \mathbb{R}^2)} \right) < \xi_{0, l} \\
	\normalsize
	\Big\|\partial_{{\sigma}} (\tilde{f}_l \circ \beta_l)  &- \partial_{{\sigma}} (\Phi^*_l \circ \beta_l) \Big\|_{C (\bar{I}; \mathbb{R}^2)}
	< \xi_{1, l}.
	\end{align*}
	The combination of the established inequalities with Theorem \ref{refcurve} shows that $\Phi^*_l \circ \beta_l$ is a reference curve for the initial curve $\tilde{f}_k \circ \beta_k$: There exists a regular $C^1$-reparametrization $\varphi_k: \bar{I} \rightarrow \bar{I}$ and a function $\rho_{k,0}: \bar{I} \rightarrow (-d, d)$ of class $C^1$, such that 
	\begin{align}
	\tilde{f}_k \circ \beta_k (\varphi_k(\sigma)) = \Phi^*_l \circ \beta_l(\sigma) + \rho_{k,0}(\sigma) (n_{\Lambda_l}(\sigma) + {\cot{\alpha}} \eta(\sigma) \tau_{\Lambda_l} (\sigma)), && \textrm{ for } \sigma \in \bar{I}
	\label{dars}
	\end{align}
	where $\Lambda_l = \Phi^*_l \circ \beta_l(\bar{I})$. Moreover, by Condition \ref{condi2} in Definition \ref{refcurvedef}, $\rho_{k,0}$ satisfies the bounds \eqref{small} and the bound
	\begin{align}
	\|\rho_{k,0}\|_{W_2^{\gamma}(I)} \leq C\left(\alpha, \Phi^*_l \circ \beta_l, \eta, \big\|\tilde{f}_k \big\|_{W_2^{\gamma}(I; \mathbb{R}^2)}\right), 
	\label{inibd}
	\end{align}
	which are required in the short time existence result Theorem \ref{local}. Consequently, we obtain by \eqref{boundf} and the fact that $S$ is a finite set that
	\begin{align*}
	\|\rho_{k,0}\|_{W_2^{\gamma}(I)} \leq \max_{l \in S} C\left(\alpha, \Phi^*_l \circ \beta_l, \eta, C_* \right) && \textrm{ for all } k \in \mathbb{N}.
	\end{align*}
	We note that the time of existence $T$ in Theorem \ref{local} is determined by $\alpha$, the reference curve $\Phi^*_l \circ \beta_l$, the coordinates $\eta$, and the constants $R_1$ and $R_2$, where
	\begin{align*}
	\|\rho_0\|_{X_{\mu}} \leq R_1 \; \textrm{ and } \; \left\|\mathcal{L}^{-1} \right\| \leq R_2.
	\end{align*}
	We recall that $\mathcal{L}^{-1}$ depends on the reference curve $\Phi_l^* \circ \beta_l$ and also on the initial curve $\rho_{k, 0}$. Since we only need finitely many reference curves to be able to represent the initial curves, it just remains to prove that for each reference curve $\Phi_l^* \circ \beta_l$ there exists a constant $C>0$ such that
	\begin{align}
	\left\|\mathcal{L}^{-1}  (\rho_{k, 0}) \right\|< C
	\label{claim}
	\end{align}
	for each $\rho_{k, 0} \in W^\gamma_2(I)$, which corresponds to an $\tilde{f}_k \circ \beta_k \in B^{W^\gamma_2(I; \mathbb{R}^2)}(\tilde{f}_l \circ \beta_l , \nicefrac{\min_{i=1, 2} \delta_{i, l}}{2 C})$. In the following, the set of those $\rho_{k, 0} \in W^\gamma_2(I)$ is denoted by $M_l$. By the compact embedding 
	\begin{align*}
	W^\gamma_2(I) \hookrightarrow W^{\bar{\gamma}}_2(I) && \textrm{ for } \gamma > \bar{\gamma} > \nicefrac{3}{2}, 
	\end{align*} 
	which is proven like in \eqref{compe}, we observe that the set $\overline{M_l}^{W^{\bar{\gamma}}_2(I)}$ is compact in $W^{\bar{\gamma}}_2(I)$. By direct calculations, it follows that the operator $\mathcal{L}(\rho_{k, 0})$ depends continuously on $\rho_{k, 0} \in W^{\bar{\gamma}}_2(I)$ fulfilling the bounds \eqref{small}. Thus, we obtain by a Neumann series argument that for each $\rho_{k, 0} \in M_l$ there exists a $\delta(\rho_{k, 0}) > 0$, such that for all $\rho \in B_0^{W^{\bar{\gamma}}_2(I)}(\rho_{k, 0}, \delta(\rho_{k, 0}))$,
	\small
	\begin{align*}
	\hspace{-0,1 cm} B_0^{W^{\bar{\gamma}}_2(I)}(\rho_{k, 0}, \delta(\rho_{k, 0})) := \left\{ \rho \in W^{\bar{\gamma}}_2(I) : \partial_{\sigma} \rho (\sigma) = 0  \textrm{ for } \sigma = 0, 1 \textrm{ and } \|\rho_{k, 0} - \rho\|_{W^{\bar{\gamma}}_2(I)} < \delta(\rho_{k, 0}) \right\},
	\end{align*}
	\normalsize
	the operator $\mathcal{L}(\rho) \in L(\mathbb{E}_{\mu, T}; \mathbb{E}_{0, \mu} \times \tilde{\mathbb{F}}_{\mu} \times X_{\mu})$ is invertible with
	\begin{align*}
	\left\| \mathcal{L}^{-1}(\rho) \right\| \leq 2 \left\| \mathcal{L}^{-1}(\rho_{k, 0}) \right\|.
	\end{align*}
	Moreover, we can cover $\overline{M_l}^{W^{\bar{\gamma}}_2(I)}$ by the union of $B_0^{W^{\bar{\gamma}}_2(I)}(\rho_{k, 0}, \delta(\rho_{k, 0}))$ with $k \in M_l$.
	By compactness of $\overline{M_l}^{W^{\bar{\gamma}}_2(I)}$, there exists a finite set $S_l \subset \mathbb{N}$ such that 
	\begin{align*}
	\overline{M_l}^{W^{\bar{\gamma}}_2(I)} \subset \bigcup_{S_l \subset \mathbb{N}} B^{W^{\bar{\gamma}}_2(I)}(\rho_{k, 0}, \delta(\rho_{k, 0})). 
	\end{align*}
	Consequently, we deduce
	\begin{align*}
	\left\| \mathcal{L}^{-1}(q) \right\| \leq 2 \max_{k \in S_l} \left\| \mathcal{L}^{-1}(\rho_{k, 0}) \right\| = : C
	\end{align*}
	for all $\rho_{k, 0} \in M$, cf.\ \eqref{claim}. 
	Thus, it makes sense to set $\tilde{T} := \min_{l \in S} T_l$. As $t_k \rightarrow T_{max}$ for ${k \to \infty}$, we can choose a sufficiently large $k \in \mathbb{N}$ such that $t_k + \tilde{T} > T_{max}$.
	
	We fix a $k \in \mathbb{N}$ with this property. Let $\rho_{k,0} \in W^{4(\mu - \nicefrac{1}{2})}_2(I)$ be the height function over $\Phi^*_l \circ \beta_l$, $l \in S$, which corresponds to $\tilde{f}_k \circ \beta_k$. By the short time existence result, Theorem \ref{local}, we obtain for the initial datum $\rho_{k,0} \in W^{4(\mu - \nicefrac{1}{2})}_2(I)$ a solution 
	\begin{align*}
	\rho: [0, \tilde{T}) \times I \rightarrow (-d, d),  \;	(t, x) \mapsto \rho(t, x),
	\end{align*}
	such that $\rho \in W^1_{2} ([0, \tilde{T}); L_{2} (I) ) \cap L_{2, loc} ([0, \tilde{T}); W^{4}_{2} (I))$ and $\rho(0, \cdot) = \rho_{k,0}$. 
	This implies that for
	\begin{align*}
	\tilde{f}(t, \sigma) &:= \Phi^*_l \circ \beta_l(\sigma) + \rho(t, \sigma) (n_{\Lambda_l}(\sigma) + {\cot{\alpha}} \eta(\sigma) \tau_{\Lambda_l} (\sigma))
	\end{align*}
	the following holds true:
	\begin{enumerate}
		\item $\tilde{f} \in W^1_{2, \mu} \left([0, \tilde{T}); L_{2, \mu} (I; \mathbb{R}^2) \right) \cap L_{2} \left([0, \tilde{T}); W^{4}_{2} (I; \mathbb{R}^2) \right)$,
		\item $\tilde{f}$ fulfills \eqref{0}-\eqref{3blow} and there exists a regular $C^1$-reparametrization $\varphi_k: \bar{I} \rightarrow \bar{I}$ such that $\tilde{f}(0, \sigma) 
		= \tilde{f}_k \circ \beta_k (\varphi_k(\sigma))$ for all $\sigma \in \bar{I}$, cf.\ \eqref{dars},
		\item $\tilde{f}(t, \cdot)$ is for each $t \in [0, \tilde{T})$ a regular parametrization of the curve $\tilde{f}(t, \bar{I})$.
	\end{enumerate}
	
	\hspace{-0.35cm}\textit{\uline{Step 3:} Extension of the original solution\\}
	It remains to show that we can extend the original solution $f$ beyond $T_{max} < \infty$, which was assumed to be maximal. To this end, we want to translate $\tilde{f}$ by $f(t_k, 0)$, as $\tilde{f}(t_k, 0) + f(t_k, 0) = f(t_k, 0)$. By Lemma \ref{unab1}, we have
	\begin{align*}
	W^1_{2, \mu, loc} \left([0, T); L_{2} (I; \mathbb{R}^2) \right) &\cap L_{2, \mu, loc} \left([0, T); W^{4}_{2} (I; \mathbb{R}^2) \right) \\
	& \hookrightarrow  BUC \left([0, T - \epsilon], W_{2}^{4 (\mu - \nicefrac{1}{2})}(I; \mathbb{R}^2) \right) && \textrm{ for each } 0 < \epsilon < T.
	\end{align*}
	Consequently, it follows by Remark \ref{afterunab}
	\begin{align*}
	|f(t_l, 0) | \leq \| f(t_l, \cdot) \|_{C(\bar{I}; \mathbb{R}^2)} \leq C \| f(t_l, \cdot) \|_{W_{2}^{4 (\mu - \nicefrac{1}{2})}(I; \mathbb{R}^2)}.
	\end{align*}
	We notice the flow is invariant under translation, i.e.\ if $f: [0, T)  \times \bar{I} \rightarrow \mathbb{R}^2$ is strong solution with $f(0, \bar{I}) = f_0(\bar{I})$, then $f_h (t, \sigma) := f (t, \sigma) + (h, 0)^T$ is a strong solution with $f_h(0, \bar{I}) = f_0(\bar{I}) + (h, 0)^T$.
	Therefore, $\tilde{f} + f(t_k, 0)$ fulfills properties analogous to $\tilde{f}$, see the previous step of the proof. By concatenating the "old" part of the solution for $t \in [0, t_k]$ and the new one for $t \in [t_k, t_k + \tilde{T}) $ at $t = t_k$, we obtain 
	an extension of the original solution as $t_k + \tilde{T} > T_{max}$, which contradicts the maximality of the original solution. Thus, the assumption \eqref{kb} on the curvature cannot be true. As the sequence $(t_l)_{l \in \mathbb{N}}$ was arbitrary, the claim is proven.

\footnotesize
\section*{Acknowledgement} 
The results of this paper are part of the second author's PhD Thesis, which was supported by the DFG through the Research Training Group GRK 1692 "Curvature, Cycles, and Cohomology" in Regensburg. The support is gratefully acknowledged.


\begin{thebibliography}{10}
	
	\bibitem{butzprint1}
	H.~Abels and J.~Butz.
	\newblock The curve diffusion flow with a contact angle.
	\newblock {\em Preprint arXiv 1810.01502}, Oct. 2018 
	
	\bibitem{alt}
	H.~W. Alt.
	\newblock {\em Linear functional analysis}.
	\newblock Universitext. Springer-Verlag London, Ltd., London, 2016.
	\newblock An application-oriented introduction, Translated from the German
	edition by Robert N\"urnberg.
	
	\bibitem{escher}
	H.~Amann and J.~Escher.
	\newblock {\em Analysis. {III}}.
	\newblock Birkh\"auser Verlag, Basel, 2009.
	\newblock Translated from the 2001 German original by Silvio Levy and Matthew
	Cargo.
	
	\bibitem{danchin}
	H.~Bahouri, J.-Y. Chemin, and R.~Danchin.
	\newblock {\em Fourier Analysis and Nonlinear Partial Differential Equations}.
	\newblock Grundlehren der mathematischen Wissenschaften 343. Springer-Verlag
	Berlin Heidelberg, 1 edition, 2011.
	
	\bibitem{butzdiss}
	J.~Butz.
	\newblock {\em The Curve Diffusion Flow with a Contact Angle}.
	\newblock PhD Thesis, 2018.
	\newblock https://epub.uni-regensburg.de/37705/.
	
	\bibitem{cahnelliottnovick}
	J.~W. Cahn, C.~M. Elliott, and A.~Novick-Cohen.
	\newblock The {C}ahn-{H}illiard equation with a concentration dependent
	mobility: motion by minus the {L}aplacian of the mean curvature.
	\newblock {\em European J. Appl. Math.}, 7(3):287--301, 1996.
	
	\bibitem{chou}
	K.-S. Chou.
	\newblock A blow-up criterion for the curve shortening flow by surface
	diffusion.
	\newblock {\em Hokkaido Math. J.}, 32(1):1--19, 2003.
	
	\bibitem{acqualinpozzi}
	A.~Dall'Acqua, C.-C. Lin, and P.~Pozzi.
	\newblock Evolution of open elastic curves in {$\Bbb{R}^n$} subject to fixed
	length and natural boundary conditions.
	\newblock {\em Analysis (Berlin)}, 34(2):209--222, 2014.
	
	\bibitem{acquapozzi}
	A.~Dall'Acqua and P.~Pozzi.
	\newblock A {W}illmore-{H}elfrich {$L^2$}-flow of curves with natural boundary
	conditions.
	\newblock {\em Comm. Anal. Geom.}, 22(4):617--669, 2014.
	
	\bibitem{acquapozzispener}
	A.~Dall'Acqua, P.~Pozzi, and A.~Spener.
	\newblock The {L}ojasiewicz-{S}imon gradient inequality for open elastic
	curves.
	\newblock {\em J. Differential Equations}, 261(3):2168--2209, 2016.
	
	\bibitem{dziuk}
	G.~Dziuk, E.~Kuwert, and R.~Sch\"atzle.
	\newblock Evolution of elastic curves in {$\Bbb R^n$}: existence and
	computation.
	\newblock {\em SIAM J. Math. Anal.}, 33(5):1228--1245, 2002.
	
	\bibitem{eschmaysim}
	J.~Escher, U.~F. Mayer, and G.~Simonett.
	\newblock The surface diffusion flow for immersed hypersurfaces.
	\newblock {\em SIAM J. Math. Anal.}, 29(6):1419--1433, 1998.
	
	\bibitem{garckenovick}
	H.~Garcke and A.~Novick-Cohen.
	\newblock A singular limit for a system of degenerate {C}ahn-{H}illiard
	equations.
	\newblock {\em Adv. Differential Equations}, 5(4-6):401--434, 2000.
	
	\bibitem{lin}
	C.-C. Lin.
	\newblock {$L^2$}-flow of elastic curves with clamped boundary conditions.
	\newblock {\em J. Differential Equations}, 252(12):6414--6428, 2012.
	
	\bibitem{lunardi}
	A.~Lunardi.
	\newblock {\em Analytic semigroups and optimal regularity in parabolic
		problems}.
	\newblock Modern Birkh\"auser Classics. Birkh\"auser/Springer Basel AG, Basel,
	1995.
	
	\bibitem{meydiss}
	M.~Meyries.
	\newblock {\em Maximal regularity in weighted spaces, nonlinear boundary
		conditions, and global attractors}.
	\newblock PhD Thesis, 2010.
	\newblock https://publikationen.bibliothek.kit.edu/1000021198.
	
	\bibitem{meyries_inter}
	M.~Meyries and R.~Schnaubelt.
	\newblock Interpolation, embeddings and traces of anisotropic fractional
	sobolev spaces with temporal weights.
	\newblock {\em J. Funct. Anal.}, 262(3):1200--1229, 2012.
	
	\bibitem{meyries}
	M.~Meyries and R.~Schnaubelt.
	\newblock Maximal regularity with temporal weights for parabolic problems with
	inhomogeneous boundary conditions.
	\newblock {\em Math. Nachr.}, 285(8-9):1032--1051, 2012.
	
	\bibitem{mullins}
	W.~W. Mullins.
	\newblock Theory of thermal grooving.
	\newblock {\em Journal of Applied Physics}, 28(3):333--339, 1957.
	
	\bibitem{pruess}
	J.~Pr\"uss and G.~Simonett.
	\newblock {\em Moving interfaces and quasilinear parabolic evolution
		equations}, volume 105 of {\em Monographs in Mathematics}.
	\newblock Birkh\"auser, Basel, 2016.
	
	\bibitem{renardy}
	M.~Renardy and R.~C. Rogers.
	\newblock {\em An introduction to partial differential equations}, volume~13 of
	{\em Texts in Applied Mathematics}.
	\newblock Springer-Verlag, New York, second edition, 2004.
	
	\bibitem{triebel}
	H.~Triebel.
	\newblock {\em Interpolation theory, function spaces, differential operators}.
	\newblock Johann Ambrosius Barth, Heidelberg, second edition, 1995.
	
	\bibitem{wheelerwheeler}
	G.~{Wheeler} and V.-M. {Wheeler}.
	\newblock {Curve diffusion and straightening flows on parallel lines}.
	\newblock {\em Preprint arXiv 1703.10711}, Mar. 2017.
	
	\bibitem{zeidler}
	E.~Zeidler.
	\newblock {\em Nonlinear functional analysis and its applications. {I}}.
	\newblock Springer-Verlag, New York, 1986.
	\newblock Fixed-point theorems, Translated from the German by Peter R. Wadsack.
	
\end{thebibliography}

\end{document}